\documentclass[12pt]{amsart}
\usepackage[english]{babel}
\usepackage[applemac]{inputenc}
\usepackage[T1]{fontenc}
\usepackage{bbm}
\usepackage{float, verbatim}
\usepackage{amsmath}
\usepackage{amssymb}
\usepackage{amsthm}
\usepackage{amsfonts}
\usepackage{graphicx}
\usepackage{mathtools}
\usepackage{enumitem}
\usepackage[pagebackref]{hyperref}
\usepackage{color}
\usepackage{tikz}
\usepackage{xcolor}
\usepackage{pgfplots}
\usepackage{caption}
\usepackage{subcaption}
\usepackage{enumerate}
\usepackage{url}

\hypersetup{
	colorlinks=true,
	linkcolor=blue,
	filecolor=magenta,      
	urlcolor=cyan,
	citecolor= red,
}

\usepackage[final]{showlabels}

\newcommand{\ubox}{\overline{\dim}_{\mathrm{B}}}

\newcommand{\Haus}{\dim_{\mathrm{H}}}

\newtheorem*{thm*}{Theorem}
\newtheorem*{conj*}{Conjecture}
\newtheorem*{ques*}{Question}
\newtheorem*{rem*}{Remark}
\newtheorem*{defn*}{Definition}
\newtheorem*{mainques*}{Main questions}

\newtheorem{thmx}{Theorem}

\newtheorem{conjx}{Conjecture}

\newtheorem{remx}{Remark}

\newtheorem{thm}{Theorem}[section]

\newtheorem{lma}[thm]{Lemma}

\newtheorem{defn}[thm]{Definition}

\newtheorem{conj}[thm]{Conjecture}
\newtheorem{rem}[thm]{Remark}

\newtheorem{exm}[thm]{Example}

\begin{document}
\title{Rational points near self-similar sets}
%    author one information

\author{Han Yu}
\address{Han Yu\\
	Department of Pure Mathematics and Mathematical Statistics\\University of Cambridge\\CB3 0WB \\ UK }
\curraddr{}
\email{hy351@maths.cam.ac.uk}
\thanks{}

%    \subjclass is required.


\maketitle

\begin{abstract}
In this paper, we consider a problem of counting rational points near self-similar sets. Let $n\geq 1$ be an integer. We shall show that for some self-similar measures on $\mathbb{R}^n$, the set of rational points $\mathbb{Q}^n$ is 'equidistributed' in a sense that will be introduced in this paper. This implies that an inhomogeneous Khinchine convergence type result can be proved for those measures. In particular, for $n=1$ and large enough integers $p,$ the above holds for the middle-$p$th Cantor measure, i.e. the natural Hausdorff measure on the set of numbers whose base $p$ expansions do not have digit $[(p-1)/2].$ Furthermore, we partially proved a conjecture of Bugeaud and Durand for the middle-$p$th Cantor set and this also answers a question posed by Levesley, Salp and Velani. Our method includes a fine analysis of the Fourier coefficients of self-similar measures together with an Erd\H{o}s-Kahane type argument. We will also provide a numerical argument to show that $p>10^7$ is sufficient for the above conclusions. In fact, $p\geq 15$ is already enough for most of the above conclusions.
\end{abstract}

\maketitle
%\tableofcontents
\allowdisplaybreaks

\section{Introduction}
\subsection{Well approximable numbers with missing digits--A question of Mahler}\label{M}
Results in this paper are  closely related to a problem asked by Mahler, see \cite{Mahler84}.
\begin{ques*}[Mahler's problem]
	How close can irrational elements of Cantor’s set be approximated by rational numbers
	
	(i) in Cantor’s set, and
	
	(ii) by rational numbers not in Cantor’s set?
\end{ques*}
Here, Cantor's set can be any missing digits set, for example, the middle third Cantor set, see below and Section \ref{Missing Digits Measures}. We concentrate on the question (ii) in this paper. In order to provide more precise information, we first introduce the notion of well approximable numbers.
\begin{defn*}
Let $n\geq 1$ be an integer. Let $\gamma\in [0,1]^n.$ Let $\psi:\mathbb{N}\to (0,1]$ be a sequence of real numbers (approximation function). We define the set of $\gamma$ shifted $\psi$-well approximable numbers to be
\[
W(\psi,\gamma)=\{x\in\mathbb{R}^n:  \|qx-\gamma\|\leq \psi(q) \text{ infinitely often}\},
\]
where $\|a\|$ denotes the distance between $a\in\mathbb{R}^n$ and the integer lattice $\mathbb{Z}^n.$ 

When $\gamma$ is the zero vector, we simply write $W(\psi)=W(\psi,\mathbf{0}).$ Let $\nu>1, n=1$ we call the set $W_\nu=W(\psi)$ with $\psi:q\to q^{-\nu}$ to be the set of $\nu$-well approximable numbers. We also call $W_{>1}=\cup_{\nu>1}W_\nu$ to be the set of very well approximable numbers.
\end{defn*}
We briefly introduce the notion of missing digits set and measures. We will provide more information in Section \ref{Missing Digits Measures}.
\begin{defn*}
	Let $p>2$ be an integer. Let $D\subset \{0,\dots p-1\}$ be a set with at least two elements. Let $\mu_{p,D}$ be the distribution of the following random number
	\[
	\sum_{j\geq 1} a_j p^{-j},
	\]
	where $a_j,j\geq 1$ are i.i.d random variables taking each element in $D$ with equal probability. We say that $K_{p,D}=supp(\mu_{p,D})$ is a ($p$-adic) missing digits set and $\mu_{p,D}$ is a ($p$-adic) missing digits measure which is the Cantor-Lebesgue measure of $K_{p,D}.$ For example, $K_{3,\{0,2\}}$ is the middle-third Cantor set.
\end{defn*}

Let $K$ be a missing digits set. Mahler's question (ii) is thus concerning the set $K\cap W(\psi,\gamma).$ In particular, we have the following two long-standing questions. 
\begin{mainques*}\label{Question}
	Let $K$ be a missing digits set with the corresponding Cantor-Lebesgue measure $\mu.$ 
	
	(1)  Let $\psi$ be a non-increasing monotonic approximation function. What can we say about $\mu(W(\psi))$?
	
	(2) What can we say about $\Haus K\cap W_{>1}$? Furthermore, let $\nu>1.$ What can we say about $\Haus K\cap W_\nu$? 
\end{mainques*}

Both (1),(2) have attracted a great amount of attention. Before we list some known results, we first mention the following guiding conjectures which provide hypothetical answers to the above questions.

\begin{conjx}[Levesley, Salp and Velani \cite{LSV ref}]\label{LSV}
	Let $K$ be a missing digits set.  We have $\Haus K\cap W_{>1}=\Haus K.$
\end{conjx}
\begin{conjx}[Bugeaud and Durand \cite{BD16}]\label{BD}
	Let $K$ be a missing digits set. Let $\nu>1.$ We have
	\[
	\Haus K\cap W_\nu=\max\left\{\frac{\Haus K}{\nu+1},\frac{2}{\nu+1}+\Haus K-1\right\}	.\]
\end{conjx}
\begin{conjx}[Kleinbock-Lindenstrauss-Weiss \cite{KLW}]\label{KH}
	Let $K$ be a missing digits set with the corresponding Cantor-Lebesgue measure $\mu.$ Let $\psi$ be a non-increasing monotonic approximation function. Then $\mu(W(\psi))=0$ if $\sum_q \psi(q)<\infty$ and otherwise $\mu(W(\psi))=1.$ In other words, the measure $\mu$ satisfies the conclusion of Khinchine's theorem which was proved in \cite{Khintchine} for the Lebesgue measure. 
\end{conjx}	

By setting $\nu\to 1$ in Conjecture \ref{BD}, we see that Conjecture \ref{LSV} follows as a consequence. All the conjectures are open. We provide a non-exhausting list of  known results.
\begin{thm*}[Known results] Let $K$ be a missing digits set and $\mu$ be the corresponding Cantor-Lebesgue measure.
	\begin{itemize}
	\item[(1)] Levesley, Salp, Velani \cite{LSV ref}: $\Haus K\cap W_{>1}\geq \Haus K/2$.
	\item[(2)]  Bugeaud, Durand \cite{BD16}: $\Haus (K+a)\cap W_\nu\leq \max\left\{\frac{\Haus K}{\nu+1},\frac{2}{\nu+1}+\Haus K-1\right\}$ for Lebesgue almost all $a\in \mathbb{R}.$
	\item[(3)] Kleinbock, Lindenstrauss, Weiss \cite{KLW}, Pollington, Velani \cite{PV2005}: $\mu(W_{>1})=0.$ \footnote{Measures with this property are called to be extremal. Other well known examples are surface measures carried by non-degenerate manifolds, see \cite{KM98}.}
	\item[(4)] Einsiedler, Fishman, Shapira \cite{EFS11}, Simmons, Weiss \cite{SW19}: Let $\epsilon>0.$ Then we have $\mu(W(\psi))=1$ for $\psi:q\to \epsilon/q.$ \footnote{Results in \cite{EFS11}, \cite{SW19} are a bit more than just $\mu(W(\psi))=1$ for $\psi: q\to\epsilon/q.$ In fact, they are able to study the continued fractions of a generic element in self-similar sets.}
	\end{itemize}
\end{thm*}
\begin{rem*}
	Recently, Khalil and L\"{u}thi \cite{KL20} are able to prove Conjecture  \ref{KH} for 'thick enough' self-similar measures with rational parameters and the open set condition. Before that, results in (3),(4) are at the cutting edge. The aim of this paper is to prove Conjecture \ref{LSV} and partially Conjecture \ref{BD} for 'thick enough' self-similar measures. This is similar to the spirit of \cite{KL20}.
\end{rem*}
\subsection{Statements of the main results}
Now, we state the main results in this paper. In what follows, for each odd integer $p>2,$ $K_p$ is the missing digits set with base $p$ and one missing digit $(p-1)/2,$ i.e. $K_{p}=K_{p,D}$ with the digits set $D=\{0,\dots,p-1\}\setminus \{(p-1)/2\}.$ We also denote $\mu_p$ to be corresponding Cantor-Lebesgue measure. We call $K_p$, resp. $\mu_p$ to be the middle-$p$th Cantor set, resp. measure.

\begin{thmx}[Main]\label{Main}
We have the following results.

(1) Conjecture \ref{LSV} holds for $K_{15}.$

(2) The convergence part of Conjecture \ref{KH} holds for $\mu_{15},$ i.e. $\sum_{q}\psi(q)<\infty$ implies that $\mu_{15}(W(\psi))=0$ for non-increasing approximation functions $\psi.$ On the other hand, we have $\mu_{15}(W(\psi))=1$ with $\psi:q\to 1/(q\log\log q).$

(3) For the missing digits set $K_{10^7+1},$ there is a number $c>0$ such that Conjecture \ref{BD} holds for $\nu\in (1,1+c)$. 
\end{thmx}
\begin{remx}
	(1),(2),(3) holds for other missing digits measures $\mu$ as well. The general conditions are as follows, see Section \ref{Lattice} for the definition of $\dim_{l^1}$:
	
	For (1),(2), we require that $\dim_{l^1}\mu>1/2.$ 
	
	For (3), we require that $\Haus \mu\dim_{l^1}\mu>1/2.$
	
	 Our method does not apply to the middle third Cantor set $K_3.$ The divergence result in (2) is weaker than that of Khalil and L\"{u}thi, and it is only a little bit better than the previous results. However, our method is very different than that in \cite{KL20}. For this reason, we will provide a proof at the end of this paper.
\end{remx}
We note that some of the results in Theorem \ref{Main} hold also with a general inhomogeneous shift $\gamma.$
\begin{thmx}\label{MainInhomo}
	We have the following results.
	
	(1) For $K_{15},$ under the assumption of Conjecture \ref{KH}, suppose that $\sum_{q}\psi(q)<\infty.$ Then $\mu(W(\psi,\gamma))=0$ for all $\gamma\in\mathbb{R}.$
	
	(2) There is a number $c>0$ such that
	\[
	\Haus K_{15}\cap W(\psi,\gamma)\leq \frac{2}{\nu+1}+\Haus K_{15}-1
	\] 
	for all $\nu\in (1,1+c)$ and $\gamma\in\mathbb{R}.$
\end{thmx}
\begin{remx}
It is in interesting, and we are not able, to prove the inhomogeneous versions of Theorem \ref{Main}(3) and the divergence part of Theorem \ref{Main}(2).
\end{remx}
\subsection{A lattice counting method}\label{Lattice}
Our main strategy for proving Theorems \ref{Main}, \ref{MainInhomo} is via a lattice counting method which we now introduce. Let $n\geq 1$ be an integer. Let $\gamma\in [0,1]^n,\delta\in (0,1).$ We construct the following set
\[
A(\delta,Q,\gamma)=\{x\in\mathbb{R}^n: \|Qx-\gamma\|\leq \delta\}.
\]
For example, if $\gamma=0$ then $A(\delta,Q,0)$ is the set of points which are $\delta/Q$-close to rational points with denominator $Q.$ Let $\mathcal{L}_{[0,1]^n}$ be the Lebesgue measure restricted to $[0,1]^n.$ It is clear that for $\delta\in (0,0.5)$
\[
\mathcal{L}_{[0,1]^n}(A(\delta,Q,\gamma))\asymp\delta^n.
\]
In fact, it is possible to compute the above Lebesgue measure exactly. Now it is interesting to see whether the above counting property holds for other probability measures as well. Towards this direction, we make the following definition.
\begin{defn}
Let $n\geq 1$ be an integer. Let $\mu$ be a Borel probability measure on $\mathbb{R}^n.$ We say that $\mu$ has
\begin{itemize}
	\item Super counting property n (or SCP(n)), if for some $\alpha>1/n$ and large enough integers $Q,$ $\mu(A(\delta,Q,\gamma))\asymp \delta^n$ as long as $\delta\gg Q^{-\alpha}.$ That is, there are numbers $c_1,c_2,M>0$ such that if $Q\geq M$ and $1>\delta\geq c_1 Q^{-\alpha}$ then
	\[
	c^{-1}_2\delta^n\leq  \mu(A(\delta,Q,\gamma))\leq c_2 \delta^n.
	\]
	\item  Good counting property n (or GCP(n)), if for some $\alpha>1/n$ and large enough integers $Q,$ $\sum_{q=Q}^{2Q}\mu(A(\delta,q,\gamma))\asymp Q\delta^n$ as long as $\delta\gg Q^{-\alpha}.$
\end{itemize}
The supreme of the possible values of $\alpha$ for the above to hold is called the super/good counting threshold of $\mu.$
\end{defn}
Clearly, SCP(n) implies GCP(n). Lattice counting properties are closely related to metric Diophantine approximations. Recently, there are many works on determining the super/good counting properties as well as the thresholds for surface measures carried by non-degenerate submanifolds, see \cite{B12},\cite{BVVZ} and \cite{H20}. In this paper, we will show that some self-similar measures in $\mathbb{R}^n$ have GCP(n). Our method relies on Fourier analysis. First, we introduce the following notion of $l^p$-dimensions.\footnote{In \cite{Sh}, a notion of $L^p$-dimension was introduced. It is not the same as the $l^p$-dimension in this paper.}
\begin{defn}
Let $\mu$ be a Borel probability measure supported on $[0,1]^n.$ Let $\hat{\mu}$ be its Fourier series. We define the $l^1$-dimension of $\mu$ to be as follows,
\[
\dim_{l^1} \mu=\sup\left\{s>0: \sum_{|\xi|\leq R} |\hat{\mu}(\xi)|\ll R^{n-s}\right\}.
\]
More generally, let $q>0.$ We define 
\[
\dim_{l^q} \mu=\sup\left\{s>0: \sum_{|\xi|\leq R} |\hat{\mu}(\xi)|^q\ll R^{n-s}\right\}.
\] 
\end{defn}
We single out the $l^1$-dimension because it will be most relevant to us. For $q=2,$ $\dim_{l^2} \mu$ is the $l^2$-dimension of $\mu.$ It is closely related to the Hausdorff dimension. In fact, for those measures that will be considered in this paper, their Hausdorff dimensions are simply equal to their $l^2$-dimensions.  We will recall the definitions in Section \ref{Preliminaries}.  Unlike the $l^2$-dimension, in general, it is rather difficult to deal with the $l^1$-dimension.\footnote{We wish to return to the exact computations of the $l^1$-dimensions of missing digits measures in a forthcoming project. For example, it is possible to show that $\dim_{l^1}\mu_3<1/2.$ In this paper, we will introduce some numerical methods for estimating the $l^1$-dimensions for some self-similar measures which are not necessarily assumed to be missing digits measures.} For this reason, we start with the case for $n=1$ and introduce the following technical definition.
\begin{defn}\label{DST}
Let $\mu$ be a Borel probability measure supported on $[0,1].$ Let $\hat{\mu}$ be its Fourier series. We say that $\mu$ is spectral if
\[
\#\{\xi\in\mathbb{Z}\cap [0,N]: \hat{\mu}(\xi)\geq |\xi|^{-\Delta}\}\ll N^{1-\lambda}
\]
for a $\Delta>1/2$ and a $\lambda>0.$ The supreme of all possible values of such $\lambda$ is called the \emph{residue dimension} of $\mu$, denoted as 
\[
\dim_{\mathrm{R}} \mu.
\] 
We say that $\mu$ is thick if $\dim_{l^2} \mu+\dim_{\mathrm{R}} \mu>1.$ We denote $\dim_{ST}\mu$ to be the supremum of all possible values of $\min\{\Delta,(\lambda+\dim_{l^2}\mu)/2\}$. Observe that $\mu$ is spectral and thick if and only if $\dim_{ST}\mu>1/2.$ 
\end{defn} 
We will prove the following lemma in Section \ref{L^1} (Parts (1),(3)) and Section \ref{LCST} (Part (2)).
\begin{lma}\label{LL^1}
Let $\mu$ be a Borel probability measure supported on $[0,1].$ Then we have the following results.
\begin{itemize}
	\item [(1)]: $\frac{1}{2}\dim_{l^2}\mu\leq \dim_{l^1}\mu\leq \dim_{l^2}\mu;$
	\item [(2)]: If $\mu$ is spectral and thick, then $\dim_{l^1}\mu>1/2.$ In fact, $\dim_{l^1}\mu\geq \dim_{ST}\mu$;
	\item [(3)]: If $\dim_{l^1}\mu>1/2$, then $\mu$ is spectral.
\end{itemize}
\end{lma}
We will prove the following two results.
\begin{thm}\label{GCPtoKS}
Let $\mu$ be such that $\dim_{l^1}\mu>1/2$. Then $\mu$ has GCP(1). It satisfies the convergence part of Conjecture \ref{KH}.  If $\mu$ is moreover a missing digits measure, then $\mu(W(\psi,0))=1$ for $\psi:q\to 1/(q\log\log q).$
\end{thm}

\begin{thm}\label{oneEXI}
There are spectral and thick self-similar measures on $[0,1]$ whose Hausdorff dimensions are smaller than one and bigger than $1/2.$ 
\end{thm}
The Lebesgue measure on $[0,1]$ is self-similar, spectral and thick. Thus the above result says that there are other interesting spectral and thick measures. Without the thickness condition, the spectral condition for missing digits measures is not really too restrictive. In fact, we will provide a very simple-to-check condition in Theorem \ref{Spectral Measures}.  For now, we only provide the following examples.
\begin{exm}
The middle-third Cantor measure is spectral. The middle-$15$th Cantor measure is spectral and thick. There are self-similar measures with GCP(1) whose Hausdorff dimensions are larger than but close to $1/2.$
\end{exm}
In fact, for AD-regular measures $\mu$, it is possible to show that spectral and thick properties imply that $\dim_{l^2}\mu=\Haus \mu>1/2.$ On the other hand, it is possible to show that there are self-similar measures with Hausdorff dimension smaller than $1/2$ which do not have GCP(1) and thus cannot be spectral and thick, see Section \ref{GCT}. We suspect the following conjecture (which replaces the $l^1$-dimension in the statement of Theorem \ref{GCPtoKS} with the Hausdorff dimension) could be true.
\begin{conj}
Let $\mu$ be a missing digits measure on $\mathbb{R}$ with $\Haus \mu>1/2.$ Then $\mu$  has GCP(1). More generally, the above conclusion holds for self-similar measures $\mu$ with the open set condition and $\Haus \mu>1/2.$
\end{conj}

Next, we will consider $n\geq 2.$ The result is not as nice as for $n=1.$ We will prove the following theorem.
\begin{thm}\label{NEXI}
There are self-similar measures on $[0,1]^n$ which have GCP(n) and with Hausdorff dimension smaller than $n.$
\end{thm}
We will provide explicit sufficient conditions for a self-similar measure to have GCP(n). We give the following examples for now. Again, see Section \ref{Missing Digits Measures} for the definition of missing digits measures.
\begin{exm}
For each $n\geq 1,$ there is a number $P(n)\geq 3$ such that $p\geq P(n)$ implies that $p$-adic missing digits measures on $\mathbb{R}^n$ with one missing digit have GCP(n).
\end{exm}
We have not made explicit computations for $P(n).$ For example, we are not able to test whether the Cantor-Lebesgue of the Sierpinski carpet \footnote{The Sierpinski carpet is a missing digits set on $[0,1]^2$ with base $3$ and one missing digit $(1,1)$.} has GCP(2). The self-similar measures with GCPs we provide in this paper are very special. Most of them have an integer reciprocal contraction ratio and integer translations. This is not essential. In Section \ref{General scaling ratios}, we will provide some examples without rational contraction ratios.
\subsection{Further comments}
\subsubsection{$l^1$-dimension}\label{L^1}
For a Borel measure $\mu$ on $[0,1]^n$ we defined its $l^1$-dimension ($\dim_{l^1}$) to be the supreme of $s$ such that
\[
\sum_{|\xi|\leq R} |\hat{\mu}(\xi)|\ll R^{n-s}.
\] 
By Cauchy-Schwarz, we see that ($n=1$)
\[
\sum_{|\xi|\leq R} |\hat{\mu}(\xi)|^2\leq\sum_{|\xi|\leq R} |\hat{\mu}(\xi)|\leq R^{1/2} R^{(1-\dim_{l^2} s)/2}.
\]
Thus, we see that in general,
\[
\dim_{l^2}\mu\geq \dim_{l^1}\mu\geq \frac{\dim_{l^2}\mu}{2}.
\]
This proves (1) of Lemma \ref{LL^1}. It is possible that the above inequality is strict. However, it seems to be difficult to compute the $l^1$-dimensions. An interesting task is thus to determine $\dim_{l^1}\mu$ for the middle-third Cantor measure $\mu.$

For (3) of Lemma \ref{LL^1}, observe that if $\dim_{l^1}\mu>1/2,$ then we can find a number $s>1/2$ such that
\[
\sum_{|\xi|\leq R} |\hat{\mu}(\xi)|\ll R^{1-s}.
\]
We see that for each number $\Delta>0,$ we have the following estimate
\[
\#\{\xi: |\xi|\leq R, |\hat{\mu}(\xi)|\geq |\xi|^{-\Delta}\}\ll R^{1-s+\Delta}.
\]
As $1-s<1/2,$ it is possible to let $\Delta>1/2$ and $1-s+\Delta<1.$ This shows that $\mu$ is spectral and proves (3) of Lemma \ref{LL^1}.
\subsubsection{good counting thresholds}\label{GCT}
We saw that good counting property (GCP) is enough for deducing a Khinchine convergence type result. It is interesting to dig further in this direction and ask what is the good counting thresholds for self-similar measures. The $l^1$-dimension is also useful here.  In fact, it is possible to see that the good counting threshold for $\mu$ is at least $\dim_{l^1}\mu/(1-\dim_{l^1}\mu)$ if $\dim_{l^1}\mu\in (0,1).$ Thus if $\dim_{l^1}\mu$ is close to one, then the good counting threshold for $\mu$ can be very large. Intuitively, the larger the threshold is, the better behaves the measure $\mu$ in terms of lattice counting. On one extreme, observe that the good counting threshold for the Lebesgue measure is $\infty.$ On the other hand, suppose that $\mu$ is AD-regular with Hausdorff dimension $s.$ Then for small enough $r>0,$ any $r$-ball has $\mu$ measures $\gg r^s.$ Suppose that $\mu$ has GCP(1), then there is an $\alpha>1$ so that we have as $Q\to\infty,$
\[
\sum_{q=Q}^{2Q}\mu(A(\delta_Q,q,0))\ll Q\delta_Q
\]
as long as $\delta_Q\gg Q^{-\alpha}.$ Suppose that $0\in supp(\mu).$ Then we see that (the interval around $0$ is included in each $A(\delta_Q,q,0)$)
\[
\sum_{q=Q}^{2Q}\mu(A(\delta_Q,q,0))\gg Q \left(\frac{\delta_Q}{Q}\right)^s.
\]
Thus one cannot have $\alpha>s/(1-s).$ This shows that if $s\leq 1/2,$ then $\mu$ does not have GCP(1) in the case when $0$ is in the support of $\mu.$
\subsubsection{Khinchine convergence type}
It is strange that for the convergence result (Theorem \ref{Main}(2)) one needs the approximation function $\psi$ to be monotonic. Indeed, for many situations in Diophantine approximation, the convergence result follows directly from the convergence Borel-Cantelli lemma and usually the monotonicity of the approximation function is not essential. However, we note here that the convergence results in this paper rely heavily on the monotonicity of the approximation function. In fact, if we drop the monotonicity condition, then we can have an approximation function which is supported on a small set. For example, let $\psi$ be supported on $2^n.$ Let $\mu$ be the middle-third Cantor measure. The following conjecture of Velani was formulated in \cite[Conjecture 1.2]{ACY}.
\begin{conj}[Velani]
If $\{\psi(2^n),n\geq 1\}$ is non-increasing and $\sum_n \psi(2^n)<\infty,$ then for $\mu$.a.e $x\in supp(\mu),$
\[
\|2^nx\|< \psi(2^n)
\] 
at most finitely often.
\end{conj}
A  convergence result for $\mu$ without the monotonicity condition will in fact prove the above conjecture which is currently open. The following result was proved in \cite{ACY}. Here $\kappa=\Haus \mu.$
\begin{thm}[Allen-Chow-Yu]
If $\sum_{n} (2^{-\log n/(\log\log n\times\log\log\log n)}\psi(2^n)^\kappa+\psi(2^n))<\infty$, then then for $\mu$.a.e $x\in supp(\mu),$
\[
\|2^nx\|< \psi(2^n)
\]
at most finitely often.
\end{thm}
This result holds not only for the middle-third Cantor measure. In fact, it holds for any $p$-adic missing digits measure with $\log p/\log 2\not\in\mathbb{Q}.$
\begin{comment}
\subsubsection{Generalized Bernoulli convolutions}
Let $k\geq 2$ be an integer and let $\lambda\in (0,1).$ We consider the law $\nu_{\lambda,k}$ of the random sum
\[
\sum_{j\geq 1} a_j\lambda^j
\]
where $a_j,j\geq 1$ are taken to be i.i.d. random variables taking values in $0,\dots,k-1$ with equal probabilities. We pose the following conjecture.
\begin{conj}
	For each $k\geq 2,$ there is a number $c_k>0$ such that $\nu_{\lambda,k}$ has GCP(1) as long as
	\[
	\Haus \nu_{\lambda,k}\geq 1-c_k.
	\]
\end{conj}
The most interesting case of this conjecture is when $k=2.$ In Section \ref{General scaling ratios} we shall see that the above conjecture actually holds for all large enough $k.$
\end{comment}
\section{Structure of this paper}
Theorems  \ref{Good Counting to Metric}, \ref{Lattice Counting}, \ref{EK}, \ref{Spectral Measures} can be proved without using any other results in this paper. The convergence part of Theorem \ref{GCPtoKS} is proved in Section \ref{LCST} using Theorems \ref{Lattice Counting}, \ref{Good Counting to Metric}.  Theorem \ref{Jarnik1} is proved by using Theorems \ref{GCPtoKS}. Theorem \ref{oneEXI} is proved in Section \ref{Simple}. Theorem \ref{NEXI} together with the examples are proved in Section \ref{ND}. Theorems \ref{Jarnik2}, \ref{Jarnik4} are proved by using results in Section \ref{Simple}.
Theorem \ref{Main}(3) is proved with the help of Theorem \ref{Jarnik1}, \ref{Jarnik2}, \ref{Jarnik4}. Finally, the divergence part of Theorem \ref{GCPtoKS} is proved in Section \ref{divergence}. Theorems \ref{Main}, \ref{MainInhomo} are then direct consequences of the above results.

This paper can be divided naturally into two parts. The first part is from Section \ref{Fourier} to Section \ref{ND}. This part deals with lattice counting properties of self-similar measures. The key point is to prove the results in Section \ref{Lattice}. Part two consists Sections \ref{Dio}, \ref{divergence} which contain several applications of the GCP of measures to metric Diophantine approximations.

\section{Preliminaries}\label{Preliminaries}
In this section, we recall many standard notions in geometric measure theory. For more details, see \cite{Fa} and \cite{Ma1}.

\subsection{Borel-Cantelli lemma}
The following version of Borel-Cantelli lemma can be found in \cite[Proposition 2]{BDV ref}.
\begin{lma}[Borel-Cantelli]\label{Borel}
	Let $(\Omega, \mathcal{A}, m)$ be a probability space and let $E_1, E_2, \ldots \in \mathcal{A}$ be a sequence of events in $\Omega$ such that $\sum_{n=1}^{\infty}{m(E_n)} = \infty$. Then 
	\[m(\limsup_{n \to \infty}{E_n}) \geq \limsup_{Q \to \infty}{\frac{\left(\sum_{s=1}^{Q}{m(E_s)}\right)^2}{\sum_{s,t=1}^{Q}{m(E_s \cap E_t)}}}.\]
	If $\sum_{n=1}^{\infty}{m(E_n)} < \infty,$ then $m(\limsup_{n \to \infty}{E_n})=0.$
\end{lma}
\subsection{Fourier series and transforms}
Let $n\geq 1$ be an integer. Let $\mu$ be a Borel probability measure on $[0,1]^n$. With possibly an abuse of the notation, we use $\hat{\mu}$ to denote both its Fourier series and transform. If we consider $\mu$ as being a measure on $\mathbb{R}^n/\mathbb{Z}^n$ then for $\xi\in\mathbb{Z}^n$ we have
\[
\hat{\mu}(\xi)=\int_{[0,1]^n} e^{-2\pi i (x,\xi)}d\mu(x),
\]
where $(x,\xi)$ is the standard Euclidean inner product. Now, we view $\mu$ as a measure on $\mathbb{R}^n,$ let $\xi\in\mathbb{R}^n,$ we also have the following 
\[
\hat{\mu}(\xi)=\int_{\mathbb{R}^n} e^{-2\pi i (x,\xi)}d\mu(x).
\]
Thus if $\mu$ is already supported on $[0,1]^n,$ the values $\hat{\mu}(\xi)$ for Fourier series and transform coincide.
\subsection{Hausdorff measures and Hausdorff dimension}
Let $n\geq 1$ be an integer. Let $F\subset\mathbb{R}^n$ be a Borel set. Let $g: [0,1)\to [0,\infty)$ be a continuous function such that $g(0)=0$. Then for all $\delta>0$ we define the following quantity
\[
\mathcal{H}^g_\delta(F)=\inf\left\{\sum_{i=1}^{\infty}g(\mathrm{diam} (U_i)): \bigcup_i U_i\supset F, \mathrm{diam}(U_i)<\delta\right\}.
\]
The $g$-Hausdorff measure of $F$ is
\[
\mathcal{H}^g(F)=\lim_{\delta\to 0} \mathcal{H}^g_{\delta}(F).
\]
When $g(x)=x^s$ then $\mathcal{H}^g=\mathcal{H}^s$ is the $s$-Hausdorff measure and Hausdorff dimension of $F$ is
\[
\Haus F=\inf\{s\geq 0:\mathcal{H}^s(F)=0\}=\sup\{s\geq 0: \mathcal{H}^s(F)=\infty          \}.
\]

We need two results which help us to estimate the Hausdorff dimension of a set. The first one provides an upper bound, see \cite[Lemma 2.1]{HS18}. 
\begin{lma}[Hausdorff-Cantelli Lemma]\label{HC}
Let $F\subset\mathbb{R}^n$ be a Borel set. Let $\{B_i\}_{i\geq 1}$ be a countable covering system for $F$, i.e. for each $x\in F$ and $r>0,$ there is an $i$ such that $B_i\subset B(x,r),$ the closed ball centred at $x$ with radius $r.$ Suppose that for an $s\in (0,n],$
\[
\sum_i diam(B_i)^s<\infty,
\]
then
\[
\Haus \limsup_{i\to\infty} B_i\leq s.
\]
\end{lma}
Next we have the following result for providing a lower estimate for the Hausdorff dimension, see \cite[LEMMA on page 975]{BV}.
\begin{lma}[Mass distribution principle]\label{MDP}
Let $F\subset\mathbb{R}^n$ be a Borel set. Let $\nu$ be a Borel probability measure supported on $F$, i,e, $\nu(F)=1.$ Suppose that for some $s\in (0,n],$ $\nu$ is an $s$-Frostman measure, i.e,
\[
\nu(B_r)\ll r^s
\]
for all $r$-balls $B_r$. Then $\Haus F\geq s.$
\end{lma}
\begin{comment}
\subsection{The upper box dimension and packing dimension}
Let $F\subset\mathbb{R}^n$ be a Borel set. Let $\delta>0$ and $N_{\delta}(F)$ be the minimum amount of $\delta$-balls needed to cover $F.$ Then the upper box dimension of $F$ is
\[
\ubox F=\limsup_{\delta\to\infty} \frac{-\log N_\delta(F)}{\log \delta}.
\]
The packing dimension of $F$ is defined to be
\[
\dim_{\mathrm{P}} F=\inf\left\{\sup_i \{\ubox F_i\}: F\subset \bigcup_i F_i\right\}.
\]
In general, $\dim_{\mathrm{P}} F\geq \Haus F.$ Just like the Hausdorff dimension, the packing dimension is countably stable as well, i.e. for any countable collection of Borel sets,
\[
\dim_{\mathrm{P}} (\cup_i F_i)=\sup_i \{\dim_{\mathrm{P}} F_i\}.
\]
It is also possible to define the packing dimension via packing measures, just like what we did for defining the Hausdorff dimension. However, we decide to take the above more straightforward definition. 
\end{comment}
\subsection{Self-similar sets (without rotations), measures, open set condition}
Let $n\geq 1$ be an integer. Let $l\geq 2$ be an integer. Let $O_1,\dots,O_l\in SO(n)$, $t_1,\dots,t_l\in\mathbb{R}^n$ and $r_1,\dots,r_l\in (0,1)$ be given. Let $p_1,\dots,p_l$ be positive numbers with $\sum_{i=1}^l p_i=1.$ For $i\in\{1,\dots,l\}$, let $f_i$ be the map
\[
x\in\mathbb{R}^n\to f_i(x)=r_iO_i(x)+t_i.
\]
Then by \cite{H81}, there is a uniquely determined compact set $F$ such that
\[
\bigcup_{i=1}^l f_i(F)=F,
\]
and a uniquely determined Borel probability measure $\mu$ such that
\[
\sum_{i=1}^l p_if_i\mu=\mu
\]
where $f_i\mu$ is the pushed forwarded measure of $\mu$ by the map $f_i.$ We call $F$ and $\mu$ to be self-similar. Let $k\geq 1$ be an integer and $i_1,i_2\dots,i_k$ be a sequence of indices. Then we see that $f_{i_k}\circ\dots\circ f_{i_1}\mu$ is a scaled copy of $\mu.$ We call such a scaled copy a branch of $\mu.$

We say that $F$ and $\mu$ has \emph{open set condition} (OSC) if there is a non-empty open set $K$ such that $f_i(K)\subset K$ for all $i\in\{1,\dots,l\}$ and for two distinct indices $i,j\in\{1,\dots,l\},$
\[
f_i(K)\cap f_j(K)=\emptyset.
\]

We say that $F,\mu$ are without rotations if $O_1,\dots,O_l$ are equal to the identity matrix. Without further noticing, self-similar sets, measures in this paper are assumed to be without rotations.
\subsection{Fourier coefficients of self-similar measures}
Fourier coefficients of self-similar measures have pleasant structures. In the case when the linear part are the same, we can write down the formula explicitly. Let $\mu$ be a self-similar measure in $R^n,n\geq 1$ with contraction ratio $r\in (0,1)$ and the linear parts (the rotations $O_i's$) are the identity. Suppose further that $l\geq 2,$ and the translations are $t_1,\dots t_l$, the probability vector is $(p_1,\dots,p_l).$ Then we have the following formula,
\[
\hat{\mu}(\xi)=\int_{\mathbb{R}^n}e^{-2\pi i (x,\xi)}d\mu(x)=\prod_{j=0}^{\infty} \sum_{s=1}^l p_s e^{-2\pi i (t_s,r^j \xi)}.
\]
In general, when the linear parts are not the same, one can still write down the Fourier coefficient as the expectation of a random walk, see for example \cite{VY20}.

\subsection{$l^2$-dimension and AD-regularity}
Let $n\geq 1$ and $\mu$ be a compacted supported Borel probability measure on $[0,1]^n.$ We recall that the $l^2$-dimension of $\mu$ ($\dim_{l^2}\mu$) is the supreme of all values $s>0$ such that for large enough $R>0,$
\[
\sum_{|\xi|\leq R} |\hat{\mu}(\xi)|^2 \ll R^{n-s},
\]
where $\hat{\mu}$ is the Fourier series of $\mu.$ Let $\mu$ be a self-similar measure. Then we can define the Hausdorff dimension of $\mu$ ($\Haus \mu$) to be the value 
\[
\lim_{\delta\to 0} \frac{\log \mu(B(x,\delta))}{\log \delta}
\]
for $\mu$.a.e $x.$ This value is well defined, see \cite{FH09}. It is simple to check that
\[
\dim_{l^2}\mu\leq \Haus \mu.
\]
Let $s>0.$ We say that $\mu$ is $s$-dimensional AD-regular if for all $x\in supp(\mu)$ and small enough (in a manner that does not depend on $x$) $R>0$ 
\[
cR^{s}\leq \mu(B(x,R))\leq CR^{s},
\]
where $c,C>0$ are constants.

Let $F$ be a self-similar set (constructed with the data in the previous section) with OSC.  Let $s$ be the solution of 
\[
\sum_{i=1}^l r^s_i=1.
\]
Suppose that $p_i=r^s_i.$ Then it is possible to check that the self-similar measure $\mu$ is $s$-dimensional AD-regular and
\[
\dim_{l^2}\mu=\Haus \mu=s.
\]
In general, under OSC, we have
\[
\dim_{l^2}\mu\geq \min_{i\in\{1,\dots,l\}}\left\{\frac{\log p_i}{\log r_i}\right\}.
\]
This is in general smaller than $\Haus \mu.$ For $n=1,$ we can say something more when OSC is not presented. Let $\mu$ be a self-similar measure and let $q> 1$  be an integer. We define $T(\mu,q)$ to be the unique solution of
\[
\sum_i p^{q}_ir_i^{-T(\mu,q)}=1.
\]
We write $D(\mu,q)=T(\mu,q)/(q-1)$ and let $s=\min\{\lim_{q\to\infty} D(\mu,q),1\}.$ Let $\epsilon>0.$ Then as long as $\mu$ has exponential separation condition (a condition that is weaker than OSC), $\mu(B(x,R))\ll r^{s-\epsilon}$ uniformly for $x\in\mathbb{R}$ and small enough $r>0,$ see \cite{Sh}. From here, it is possible to see that
\[
\dim_{l^2} \mu\geq s.
\]
\subsection{Missing digits sets and measures}\label{Missing Digits Measures}
Here, we introduce a special class of self-similar measures with OSC. Let $n\geq 1$ be an integer. Let $p$ be an integer. We will consider self-similar measures with scaling ratio $1/p.$ Let $1<r<p^n$ be an integer. We choose a probability vector $p_i=1/r,i\in\{1,\dots,r\}$ and integral translations $a_1,\dots,a_r\in [0,p-1]^n.$ From the above data, it is possible to uniquely construct a self-similar measure $\mu.$ Such a measure is called to be a missing digits measure. Let $K=supp(\mu).$ This self-similar set $K$ called to be a missing digits set with $\mu$ being its corresponding Cantor-Lebesgue measure. If  Let $n=1,$ and $p>1$ be an odd number. The middle-$p$th measure is the $p$-adic missing digits measure with $p-1$ translations $\{0,\dots,p-1\}\setminus\{[(p-1)/2]\}$. The Fourier transform of a missing digits measure is
\[
\hat{\mu}(\xi)=\int_{\mathbb{R}^n}e^{-2\pi i (x,\xi)}d\mu(x)=\prod_{j=0}^{\infty} \sum_{s=1}^r p_s e^{-2\pi i (a_s,p^{-j} \xi)}.
\]
Let $f(\xi)= \sum_{s=1}^r p_s e^{-2\pi i (a_s,\xi)}.$ Then we see that $f(\xi)\to 1$ as $\xi\to (0,0,\dots,0).$ Moreover, it is possible to see that there is a constant $C>0$ such that
\[
\prod_{j\geq 0} |f(p^{-j}\xi)|\leq C
\]
holds for all $\xi\in [0,1]^n.$ This shows that $\hat{\mu}(\xi)$ can be effectively estimated from above up to a multiplicative constant by computing
\[
\prod_{j=0}^{k} \sum_{s=1}^r p_s e^{-2\pi i (a_s,p^{-j} \xi)}
\]
where $k$ is the smallest integer such that $|\xi|\in [0,p^k]^n.$ The advantage of this consideration is that we only need to deal with a finite product rather than an infinite product.
\subsection{Renormalized self-similar measures}
Let $n\geq 1$ be an integer. Let $\mu$ be a self-similar measure on $\mathbb{R}^n$. One can rescale and translate $\mu$ in a way such that $(0,\dots,0)$ is one of the translations. Moreover, if it is possible to arrange that all the translations are integral, i.e., elements in $\mathbb{Z}^n,$ then we say that $\mu$ has grid structure. For example, this is the case if the translations of the original self-similar system are all rational points, i.e. elements in $\mathbb{Q}^n.$ Missing digits measures defined previously have grid structure.

Next, since it is possible to translate $\mu$ such that $\mu([0,1]^n)>0,$ We can then find a branch of $\mu$ which is supported in $[0,1]^n.$ Self-similarity of $\mu$ says that this branch $\mu'$ is a rescaled copy of $\mu.$ In this way, lattice counting properties of $\mu$ can be obtained by lattice counting properties of $\mu'.$ For example, under OSC, $\mu$ is actually a sum of disjointly supported translated copies of $\mu'.$ For this reason, we only consider self-similar measures supported on $[0,1]^n.$ 

The reason for this renormalization is that sometimes we need to consider $\mu$ as a measure on $\mathbb{R}^n/\mathbb{Z}^n$ and we can the apply the theory of Fourier series rather than that of Fourier transform. One advantage here is that it is often more convenient to deal with sums than integrals. Another advantage is that we save some energy by not introducing the theory of tempered distributions and their Fourier transforms. Of course, our method can be adapted for directly dealing with Fourier transform.
\subsection{Asymptotic symbols}
We will use the standard $\Theta, O,o$ notations as well as the standard $\ll, \gg, \asymp$ notations. Sometimes, we add subscripts to indicates the variable, for example,
\[
o_{r\to 0}(1)
\]
indicates a value changing with $r$ and tends to $0$ as $r\to 0.$ Similarly $A\asymp_{t\to\infty}B$ indicates two values $A,B$, both changing with $t$, such that
\[
\lim_{t\to\infty} |A/B|\in (0,\infty).
\]
\section{Fourier analysis: basics}\label{Fourier}
In this section, we will prove the following fundamental and standard counting estimate via Fourier analysis. Recall that, given a Borel measure $\mu$ on $[0,1]^n,$ its Fourier transform is defined as
\[
\xi\in\mathbb{R}^n\to\hat{\mu}(\xi)=\int e^{-2\pi i (x,\xi)}d\mu(x),
\]
where $(x,\xi)$ is the Euclidean inner product of $x,\xi.$ Let $\phi$ be a Schwartz function, we denote its Fourier transform as
\[
\hat{\phi}(\xi)=\int_{\mathbb{R}^n} \phi(x)e^{-2\pi i (x,\xi)}dx.
\]
Let $g$ be a smooth $\mathbb{Z}^n$ periodic function. We define its Fourier series as
\[
\xi\in\mathbb{Z}^n\to\hat{g}(\xi)=\int_{[0,1]^n} e^{-2\pi i (x,\xi)} g(x) dx.
\]
\begin{thm}\label{Lattice Counting}
Let $n\geq 1$ be an integer. Let $\mu$ be a compactly supported Borel probability measure on $[0,1]^n.$ Let $\delta\in (0,1),$ $Q,K\in\mathbb{N}, \gamma\in [0,1]^n.$ Then there are constants $c_1,c_2>0$ such that
\[
\mu(A(\delta,Q,\gamma))\leq c_1\delta^n\left( 1+O_{Q\to\infty}(\sum_{Q|\xi,\xi\neq 0,|\xi|\leq 2Q/\delta} |\overline{\hat{\mu}(\xi)}|)\right)
\]
and
\[
\mu(A(\delta,Q,\gamma))\geq c_2\delta^n \left(1+O_{Q\to\infty}(\sum_{Q|\xi,\xi\neq 0,|\xi|\leq KQ/\delta}  |\overline{\hat{\mu}(\xi)}|)\right)+O_{K\to\infty}(K^{-N}).
\]
The implied constants in the $O_{Q\to\infty}(.)$ terms are absolute and the implied constant for the $O_{K\to\infty}(.)$ term depends on $N$ only.
\end{thm}
The constants $c_1, c_2$ can be both made arbitrarily close to one by choosing the bump functions in the proof delicately. However, we will not use this in this paper so we omit the proof. Later on, when we use the lower bound of the above theorem, we can effectively ignore the $K^{-N}$ term. The way to do this is to choose a small enough $\epsilon$, $N=[1/\epsilon^2]$ and let $K=\delta^{-\epsilon}.$ For our purpose, the small enough $\epsilon$ creates no significant difficulties. We will prove this theorem in the next two subsections.
\subsection{Homogeneous lattice counting}
We first treat the homogeneous case, i.e. $\gamma=(0,\dots,0)$. We write $A(\delta,Q)$ for $A(\delta,Q,(0,\dots,0)).$ Let $Q\geq 1$ be an integer and $\delta\in (0,1).$  We are interested in $\mu(A(\delta,Q)).$  Our aim is to show that
\[
\mu(A(\delta,Q))\approx \delta^n(1+\sum_{Q|\xi,\xi\neq 0,|\xi|\leq KQ/\delta}|\overline{\hat{\mu}(\xi)}|).
\]
It is quite natural to construct the function $f_{\delta,Q}$ to be the characteristic function on the union of $\delta/Q$-balls centred at $Q$-rational points in $\mathbb{R}^n$ and then compute the integral
\[
\mu(A(\delta,Q))=\int f_{\delta,Q}(x)d\mu(x).
\]
However, it is more convenient to deal with Schwartz functions. Here, we let $\phi$ be a radial, positive valued Schwartz function with the following properties\footnote{We restrict $\phi$ according to its Fourier transform. This can be achieved by choosing a Schwartz function with certain properties and take the inverse Fourier transform. See \cite[Section 6, beginning of the proof of Theorem 1.3]{Y20b} for more details.}:
\begin{itemize}
\item $\hat{\phi}$ is compactly supported and real valued with values in $[0,1].$
\item $\hat{\phi}(\xi)=1$ for $|\xi|\leq 1.$ 
\item $\hat{\phi}(\xi)=0$ for $|\xi|>2.$
\end{itemize}
There are several quantities which are of great use:
\begin{itemize}
\item $\|\phi\|_{L^1}$
\item For each integer $N\geq 1,$ there exists a constant $C_N>0$ such that $\phi(x)\leq C_N (1+|x|)^{-N}$ for all $x\in\mathbb{R}^n.$
\end{itemize}
After choosing a bump function $\phi,$ we will use it in place of the characteristic function of the unit ball around the origin. Let $r>0$ be a real number, we then construct the $r$-scaled bump function
\[
\phi_r(x)=\phi(x/r).
\]
Thus, $\phi_r$ is roughly supported on a ball of radius $\approx r.$  Now, construct the following function \[g_{\delta,Q}(x)=\sum_{\mathbf{k}\in\mathbb{Z}^n} \phi_{\delta/Q}(x-\mathbf{k}Q^{-1}).\]
Intuitively, this is a function with spikes at each rational point with denominator $Q.$ To see the above function is well defined, observe that because $\phi$ is Schwartz, for all $x\in\mathbb{R},$
\[
\sum_{\mathbf{k}\in\mathbb{Z}^n}\phi(x+\mathbf{k})<\infty.
\]
Notice that $g_{\delta,Q}$ is not a Schwartz function any more as it is not zero on all $Q$-rational number. It is in fact $\mathbb{Z}^n/Q$-periodic. It is then easy to see that for  each $x\in\mathbb{R}^n$ such that $\|Qx\|\in [0,\delta]$ we have $g_{\delta,Q}(x)\geq c$ for a constant $c>0$ depending only on $\phi.$ From here we see that
\begin{align*}
c\mu(A(\delta,Q))\leq \int  g_{\delta,Q}(x)d\mu(x)\tag{*}
\end{align*}
As $g_{\delta,Q}$ is $\mathbb{Z}^n/Q$-periodic, it is $\mathbb{Z}^n$-periodic as well. We can view it as a function on $\mathbb{R}^n/\mathbb{Z}^n.$ Since we already assumed that $\mu$ is supported on $[0,1]^n,$  using Fourier series we have,
\[
\int g_{\delta,Q}(x)d\mu(x)=\sum_{\xi\in\mathbb{Z}^n} \hat{g_{\delta,Q}}(\xi)\overline{\hat{\mu}(\xi)}.
\]
Then we have the following steps to compute $\hat{g_{\delta,Q}},$
\begin{align*}
\tag{Fourier Series}\hat{g_{\delta,Q}}(\xi)&=&\int_{[0,1]^n} e^{-2\pi i (\xi, x)}\sum_{\mathbf{k}\in\mathbb{Z}^n} \phi_{\delta/Q}(x-\mathbf{k}/Q)dx\\
&=& \chi_{Q|\xi} Q^n\hat{\phi_{\delta/Q}}(\xi)
=\chi_{Q|\xi} Q^n  (\delta/Q)^n\hat{\phi}((\delta/Q)\xi)\\
&=& \chi_{Q|\xi} \delta^n \hat{\phi}(\delta\xi/Q).
\end{align*}
Here, $\chi_{Q|\xi}=1$ when $Q|\xi$ (i.e. $Q$ divides all components of $\xi$) and $0$ otherwise. On the second line, $\hat{\phi}_{\delta/Q}$ denotes the Fourier transform. Hence, we see that
\begin{align*}
\int g_{\delta,Q}(x)d\mu(x)=\delta^n( \|\phi\|_{L^1}+O(\sum_{Q|\xi,\xi\neq 0, |\xi|\leq 2Q/\delta} |\overline{\hat{\mu}(\xi)}|))\tag{**}.
\end{align*}
The cut-off at $2Q/\delta$ is because $\hat{\phi}$ vanishes outside the ball of radius $2$ centred at the origin. From here we see that it is of great interest to estimate the sum
\[
\sum_{Q|\xi,\xi\neq 0, |\xi|\leq 2Q/\delta} |\overline{\hat{\mu}(\xi)}|.
\]
This will be one of the main tasks later on in this paper. From (*)(**) we conclude that
\begin{align*}
\mu(A(\delta,Q))\ll \delta^n (\|\phi\|_{L^1}+O(\sum_{Q|\xi,\xi\neq 0, |\xi|\leq 2Q/\delta} |\overline{\hat{\mu}(\xi)}|)).\tag{Upper}
\end{align*}
For convenience, we have set $\hat{\phi}$ to be compactly supported. This works well for (Upper). It is in principle possible to draw a lower bound for $\mu(A(\delta,Q)).$ However, for this problem,  it is better to choose a compactly supported bump function to replace $\phi$ for the above steps. The price to pay is that the $O(.)$ term in (**) is no longer a finite sum but an infinite sum of form  $\sum C(\xi)|\hat{\mu}(\xi)|$ with $C(\xi)\to 0$ rapidly. We now give more details, let $\phi'$ be a compactly supported bump function. Then $\hat{\phi'}$ is Schwartz, but not compactly supported. For any integer $N>0,$ we have
\[
\hat{\phi'}(\xi)=O(|\xi|^{-N}).
\]
We now have the following estimate
\begin{align*}
\tag{**'} \int g_{\delta,Q}(x)d\mu(x)=\delta^n( \|\phi'\|_{L^1}+O(\sum_{Q|\xi,\xi\neq 0} |\hat{\phi'}(\delta\xi/Q)| |\overline{\hat{\mu}(\xi)}|))
\end{align*}
For the sum inside the $O(.),$ we can choose a large integer $K$ (according to $\delta$) and split the sum according to $|\xi|\leq KQ/\delta$ or not.  We use the decay for $\hat{\phi'}$ to see that
\[
\sum_{Q|\xi,\xi\neq 0, |\xi|> KQ/\delta} |\overline{\hat{\mu}(\xi)}|\leq C_N\frac{1}{K^N} \frac{1}{\delta^n},
\]
where $C_N>0$ is a constant depending on $N$ and $\phi'.$ As $\phi'$ is compactly supported, it is possible to see that there is a number $M$ (depending only on $\phi'$) such that
\begin{align*}
\tag{Lower} \mu(A(M\delta,Q))\geq M^{-1}\delta^n (1+\sum_{Q|\xi,\xi\neq 0,|\xi|\leq KQ/\delta} |\hat{\phi'}(\delta\xi/Q)| |\overline{\hat{\mu}(\xi)}|)-C_N K^{-N}.
\end{align*}
Thus the theorem is proved in the case when $\gamma=(0,\dots,0).$
\subsection{Inhomogeneous lattice counting}
Let $\gamma\in\ [0,1]^n.$ We want to consider $\mu(A(\delta,Q,\gamma)).$ The Fourier analytic method is very effective for this type of problem as long as we are only using the norms of the Fourier coefficients. More explicitly, recall the computation (Fourier Series) in the previous section. We now introduce the 'shifted' function
\[g_{\delta,Q,\gamma}(x)=\sum_{\mathbf{k}\in\mathbb{Z}^n} \phi_{\delta/Q}(x-(\mathbf{k}+\gamma)Q^{-1}).\]
Then we have $|\hat{g}_{\delta,Q,\gamma}|=|\hat{g}_{\delta,Q}|$ since the shift $\gamma$ only introduce a phase factor (with unit norm) to the Fourier coefficients. Thus results (Upper) and (Lower)  obtained in the previous section hold also with the shift $\gamma.$ Thus Theorem \ref{Lattice Counting} concludes.
\section{Power Fourier decay and the $l^1$ method}
In this section, we provide some examples of measures with super counting properties (SCP). 
\subsection{Power Fourier decay and lattice counting}\label{PFD}
Let $n\geq 1$ be an integer and $\mu$ be a Borel probability measure on $[0,1]^n.$ Suppose that $|\hat\mu(\xi)|\ll |\xi|^{-\sigma}$ for some $\sigma>0.$ The supreme of all possible values for $2\sigma$ is referred to as the Fourier dimension:
\[
\mathrm{dim}_{F} \mu=\sup\{2\sigma>0: \mu(\xi)\ll |\xi|^{-\sigma}\}.
\]
We can use Theorem \ref{Lattice Counting} directly. Let $\delta\in (0,1), Q,K\in\mathbb{N}, \gamma\in [0,1]^n.$ The task is to estimate the sum
\[\sum_{Q|\xi,\xi\neq 0,|\xi|\leq KQ/\delta}  |\overline{\hat{\mu}(\xi)}|.
\]
Since $|\hat{\mu}(\xi)|\ll |\xi|^{-\sigma},$ we see that
\[
\sum_{Q|\xi,\xi\neq 0,|\xi|\leq KQ/\delta}  |\overline{\hat{\mu}(\xi)}|\ll \left(\frac{KQ}{\delta}\right)^{-\sigma} K^n\delta^{-n}=Q^{-\sigma} \left(\frac{K}{\delta}\right)^{n-\sigma}.
\]
To estimate $\mu(A(\delta,Q,\gamma))$ from above, we can choose $K=2.$ Then, as long as
\[
\delta\geq 200 \times Q^{-\sigma/(n-\sigma)},
\]
we have
\[
\mu(A(\delta,Q,\gamma))\ll \delta^n.
\]
For a lower bound, we need to choose a large enough number $N$ and choose $K=[1+\delta^{-1}]^{2n/N}.$ This will make
$K^{-N}$ to be much smaller than $\delta^n.$ Then, as long as
\[
\delta\geq 200 \times Q^{\frac{-\sigma/(n-\sigma)}{1-2n N^{-1}}},
\]
we have
\[
\mu(A(\delta,Q,\gamma))\gg \delta^n.
\]
\subsection{Examples}
Here we provide two (well known) classes of measures with super counting property.
\subsubsection{hypersurfaces with non-vanishing curvatures}
A well known class of measures with nice power Fourier decay comes from hypersurfaces with non-vanishing curvatures. More precisely, let $M$ be a hypersurface in $\mathbb{R}^n$ with non-vanishing curvature. Let $\mu$ be the restriction of the Hausdorff measure $\mathcal{H}^{n-1}$ on this manifold $M.$ If moreover $M$ is compact, we can normalize $\mu$ to be a probability measure. It then follows (see \cite[Chapter VIII, Section 3.1, Theorem 1]{Stein}) that
\[
|\hat{\mu}(\xi)|=O(|\xi|^{-(n-1)/2}).
\]
Thus we can set $\sigma=(n-1)/2$ in the previous section. As a result, as long as
\[
\delta\gg Q^{-(n-1)/(n+1)+\epsilon}
\]
for an $\epsilon>0,$ we have
\[
\mu(A(\delta,Q,\gamma))\asymp \delta^n.
\]
This result was proved in \cite{BVVZ} and improved in \cite{H20}. Huang's argument in \cite{H20} was built on this Fourier decay property together with an induction argument on scales.
\subsubsection{Brownian motions}
Let $n\geq 2.$ Let $(\Omega,\lambda)$ be a probability space. Let $B: \Omega\times\mathbb{R}\to\mathbb{R}^n$ be the $n$-dimensional Brownian motion. It is known that for any compact subset $E$ of $\mathbb{R}$, the image $B(\omega,E)$ is a Salem set for $\lambda$.a.e. $\omega\in \Omega$ whose dimension is twice the Hausdorff dimension of $E.$ We can choose $E=[0,1].$ Then $B(\omega,E)$ is a random curve which is almost surely (on the choice of $\omega$) Salem with dimension two. This is to say, almost surely, it supports a probability measure $\mu$ such that for all $\epsilon>0,$
\[
|\hat{\mu}(\xi)|=O(|\xi|^{-1+\epsilon}).
\]
The measure $\mu$ depends on the choice of $\omega.$ In what follows, we shall fix such an $\omega$ and $\mu.$ From what we have done, we see that
$
\mu(A(\delta,Q,\gamma))\asymp \delta^n
$
as long as $\delta\gg Q^{-1/(n-1)+\epsilon}$ for an $\epsilon>0$ which can be chosen to be arbitrarily small.

\section{Spectral measures and the $l^2$ method}
In this section, we shall introduce another method that allows us to deal with some measures  without power Fourier decay. We shall see that many self-similar measures fall into this class. To illustrate the $l^2$ method, we start with the following simple observation.
\subsection{Lattice counting for AD-regular measures}\label{AD}
Let $\mu$ be an AD-regular Borel measure on $[0,1]^n.$ Let $\delta\in (0,1), Q\in\mathbb{N}, \gamma\in [0,1]^n$ be given. In this section, we will show that
\[
\sum_{q=Q}^{2Q}\mu(A(\delta,q,\gamma))\asymp \delta^n Q
\]
for $\delta\gg Q^{-(1-(n-s)/(2n-s))+\epsilon},$ where $s$ is the $l^2$-dimension of $\mu$ (which is equal to $\Haus \mu$) and $\epsilon>0$ is number that can be arbitrarily close to zero. We also require that $(n-s)<1.$ We have the following standard result (Plancherel)
\[
\sum_{|\xi|\leq R} |\mu(\xi)|^2\asymp\int_{[0,1]^n} (\mu(B(x,1/R))/(1/R)^n)^2 dx
\]
This is the $L^2$-average (under the Lebesgue measure) of the $\mu$ measure of $1/R$-balls. In general, if $\dim_{l^2} \mu>s,$ then
\begin{align*}
\sum_{|\xi|\leq R} |\mu(\xi)|^2\ll R^{n-s}\tag{L2}.
\end{align*}
Although $\mu$ will often have dimension $s< n$, it is instructive to see what happens if $s\geq n.$ In this case, $\hat{\mu}$ has bounded $l^2$ sum. This implies that $\mu$ has $L^2$-density, i.e. it can be viewed as an $L^2$-function (with respect to the Lebesgue measure). 
The estimate (L2) indicates that $\hat{\mu}(\xi)\ll |\xi|^{-s/2}$ holds for 'most of' vectors $\xi.$ We will make use of this soon.

Now, the aim is to estimate
\[
\sum_{q=Q}^{2Q}\mu(A(\delta,q,\gamma)).
\]
From Theorem \ref{Lattice Counting} in above, it is enough to estimate the following sum
\[
\sum_{q=Q}^{2Q}\sum_{q|\xi,\xi\neq 0, |\xi|\leq Kq/\delta} |\overline{\hat{\mu}(\xi)}|\leq\sum_{\xi\in\mathbb{Z}^n, |\xi|\leq 2KQ/\delta} d(\xi)|\hat{\mu}(\xi)|
\]
where $d(\xi)$ is the number of divisors of $\xi$ in $[Q,2Q].$ We know that $d(\xi)=o(|\xi|^\epsilon)$ for all $\epsilon>0.$ Therefore by Cauchy-Schwarz and (L2) we see that for large enough $Q,$
\begin{align*}
\sum_{\xi\in\mathbb{Z}^n\setminus\{\mathbf{0}\}, |\xi|\leq 2KQ/\delta} d(\xi)|\hat{\mu}(\xi)|\leq \left(\frac{2KQ}{\delta}\right)^{\epsilon} \left(Q\frac{2^nK^n}{\delta^n}\right)^{1/2} \left(\frac{2KQ}{\delta}\right)^{(n-s)/2}.\tag{C-S}
\end{align*}
For the second term, observe that the sum on LHS has roughly $Q\times (2K/\delta)^n$ many non-zero terms. Then we see that
\[
\sum_{q=Q}^{2Q}\mu(A(\delta,q,\gamma))\asymp Q\delta^n 
\]
as long as
\[
\left(\frac{2KQ}{\delta}\right)^{\epsilon} \left(Q\frac{2^nK^n}{\delta^n}\right)^{1/2} \left(\frac{2KQ}{\delta}\right)^{(n-s)/2}\ll Q,
\]
and
$
QK^{-N}\ll Q\delta^n.
$
This condition can be reformulated to the following
\[
\delta\gg Q^{1-((n-s))/(2n-s)+\epsilon'}
\]
for $\epsilon'>0$ which can be chosen to be arbitrarily small. Here $\epsilon'>0$ was introduced to counter the $O(K^{-N})$ term in Theorem \ref{Lattice Counting}. We have met this type of argument already in Section \ref{PFD}. As the value of
\[
\frac{1-(n-s)}{2n-s}
\]
is never larger than $1/n,$ we can not conclude that $\mu$ has GCP(n) or not. The step where we lose control is (C-S). More precisely, the middle term on the rightmost side is rather crude. This term represents the number of summands on the LHS. In general, one cannot hope to improve anything here. However, for example, if we know that $\hat{\mu}$ is thinly supported, i.e. $\hat{\mu}=0$ for most of the $\xi's$ in consideration, then the middle term on the rightmost side of (C-S) can take a much smaller value. Our strategy is to obtain a better analysis of the Fourier coefficients of $\mu$ and manage to have a discount for the amount of $\xi's$ for performing the step (C-S).
\subsection{Lattice counting for spectral, thick measures, proofs of Lemma \ref{LL^1}(2) and Theorem \ref{GCPtoKS}}\label{LCST}
We now want to apply the above arguments to spectral measures. We first prove Theorem \ref{GCPtoKS}. Let $n=1$ and $\mu$ be such that $\dim_{l^1}\mu>1/2.$  Let $\delta\in (0,1),Q\in\mathbb{N},\gamma\in [0,1]$ be given. Again, we try to estimate
\[
\sum_{q=Q}^{2Q} \mu(A(\delta,q,\gamma)).
\]
All the arguments in the previous section are still valid. We arrive at the following inequality
\[
\sum_{q=Q}^{2Q}\sum_{q|\xi,\xi\neq 0, |\xi|\leq Kq/\delta} |\overline{\hat{\mu}(\xi)}|\leq \sum_{\xi\in\mathbb{Z}\setminus{0}, |\xi|\leq 2KQ/\delta} d(\xi)|\hat{\mu}(\xi)|.
\]
Since $d(\xi)=o(|\xi|^{\epsilon})$ for all $\epsilon>0,$ we see that
\[
\sum_{q=Q}^{2Q}\sum_{q|\xi,\xi\neq 0, |\xi|\leq Kq/\delta} |\overline{\hat{\mu}(\xi)}|\ll (KQ/\delta)^{1-\dim_{l^1}\mu+\epsilon}
\]
for all $\epsilon>0.$ From Theorem \ref{Lattice Counting}, we see that (for any fixed number $N>0$ and $K>0$)
\[
\sum_{q=Q}^{2Q} \mu(A(\delta,q,\gamma))\asymp \delta \left(Q+(KQ/\delta)^{1-\dim_{l^1}\mu+\epsilon}\right)+O(QK^{-N}).\tag{E}
\]
Thus for $\epsilon'>0$ which can be chosen to be arbitrarily close to zero, as long as
\[
\delta\gg Q^{-\dim_{l^1}\mu/(1-\dim_{l^1}\mu)+\epsilon'},
\]
we have
\[
\sum_{q=Q}^{2Q} \mu(A(\delta,q,\gamma))\asymp Q\delta.
\]
Since $\dim_{l^1}\mu>1/2,$ we see that
\[
\frac{\dim_{l^1}\mu}{1-\dim_{l^1}\mu}>1.
\]
Thus, we see that there is a constant $c>0$ depending on the measure $\mu$ such that
\[
\sum_{q=Q}^{2Q} \mu(A(\delta,q,\gamma))\asymp Q\delta
\]
for $\delta\gg Q^{-1-c}.$ In this case, the error term
\[
\delta(KQ/\delta)^{1-\dim_{l^1}\mu+\epsilon}+O(QK^{-N})
\]
contributes at most $O(\delta Q^{1-c'})$ for some $c'>0.$  This implies that $\mu$ has GCP(1) and concludes the proof of Theorem \ref{GCPtoKS}. 

Now, we prove Part (2) of Lemma \ref{LL^1}. Let $\mu$ be a spectral measure with parameter $\Delta>1/2, \lambda\in (0,1)$ on $[0,1].$ We can decompose the sum on $\xi$ into two parts according to whether $|\hat{\mu}(\xi)|\leq |\xi|^{-\Delta}$ or not. Let $R$ be a large integer. Consider the set of integers with absolute value less than $R.$ Let $GOOD$ be the set of integers $\xi$ with $|\hat{\mu}(\xi)|\leq |\xi|^{-\Delta}$ and $BAD$ be the complement. Then we see that
\[
\sum_{\xi\in\mathbb{Z}, |\xi|\leq R} |\hat{\mu}(\xi)|\leq \sum_{\xi\in GOOD}|\xi|^{-\Delta}+(\# BAD)^{1/2} R^{(1-s)/2},
\]
where $s=\dim_{l^2}\mu$ and we have used Cauchy-Schwarz for the second term in the RHS above. We know that
\[
\#BAD\leq R^{1-\lambda}.
\]
This implies that
\[
(\# BAD)^{1/2} R^{(1-s)/2}\ll R^{(2-\lambda-s)/2}.
\]
Thus if $\lambda+s>1,$ we have
\[
\sum_{\xi\in\mathbb{Z}, |\xi|\leq R} |\hat{\mu}(\xi)|\ll R^{1-\Delta}+R^{(2-s-\lambda)/2}\ll R^{1-\Delta'}
\]
for a $\Delta'>1/2.$ This shows that $\dim_{l^1}\mu>1/2$. Moreover, from above we see that $\dim_{l^1}\mu\geq \dim_{ST}\mu.$ This  proves Lemma \ref{LL^1}(2).
\subsection{Prime denominators}\label{Prime}
Later in this paper, we will count rationals with prime denominators near self-similar sets. To do this, we want to estimate
\[
\sum_{q=Q, q\text { is prime}}^{2Q}\mu(A(\delta,q,\gamma)).
\]
Arguments in the previous section still work and it is possible to obtain the following analogue of (E):
\begin{align*}
\sum_{q=Q,q\text{ is prime}}^{2Q} \mu(A(\delta,q,\gamma))\asymp \delta \left(\pi(2Q)-\pi(Q)+(KQ/\delta)^{1-\dim_{l^1}\mu+\epsilon}\right)+O(QK^{-N}).
\end{align*}
In fact, the main term $\pi(2Q)-\pi(Q)$ counts the number of prime numbers in $(Q,2Q].$ The error terms $(KQ/\delta)^{1-\dim_{l^1}\mu+\epsilon}$ and $O(QK^{-N})$ can be obtained in exactly the same way as in the previous section.

By the prime number theorem, we see that
\[
\pi(2Q)-\pi(Q)=\frac{Q}{\log Q}+O\left(\frac{Q}{(\log Q)^2}\right).
\]
Then with the same arguments as in above, we see that if $\mu$ is spectral and thick, there is a positive number $c>0$ such that for $\delta\gg Q^{-1-c},$
\[
\sum_{q=Q,q\text{ is prime}}^{2Q} \mu(A(\delta,q,\gamma))\asymp \delta \frac{Q}{\log Q}.
\]
Indeed, the above estimate has an error term of order $O(\delta Q^{1-c'}+\delta Q/((\log Q)^2))$ for a $c'>0.$ This error term is much smaller than $\delta Q/\log Q$ for all large enough $Q.$
\subsection{First examples of self-similar measures with GCP(1)}
Here we will present measures with GCP(1) with Hausdorff dimension larger than but close to $1/2.$  We have that for spectral and thick measures $\mu,$
\[
\dim_{l^1} \mu>1/2.\]
Then we see that GCP(1) holds with threshold $\dim_{l^1}\mu/(1-\dim_{l^1}\mu)>1.$ It is in general difficult to estimate $\dim_{l^1}\mu$ unless some structural conditions are assumed for $\mu.$ In this section, we will provide some examples. 

Let $\mu$ be a Borel probability measure on $[0,1].$ Suppose that $\dim_{l^2} \mu=s>1/2.$ Consider the measure $\mu_1=\mu*\mu,$ viewed as a convolution on $\mathbb{T}=\mathbb{R}/\mathbb{Z}.$ Then the Fourier series $\hat{\mu}_1$ is the pointwise square of $\hat{\mu}.$ In particular, we have
\[
\dim_{l^1}\mu_1=\dim_{l^2}\mu=s>1/2.
\]
Let $\mu$ be the $p$-adic missing digits measure with $p=16$ and the digit set $\{0,1,2,3,4\}.$ Then $\dim_{l^1}\mu_1=\dim_{l^2}\mu=\Haus \mu=\log 5/\log 16>1/2.$ Now $\mu_1$ is not a missing digits measure. However, it is a self-similar measure with scaling ratio $1/16$, translations $\{0,1,2,3,4,5,6,7,8\}$ and probability vector
\[
\left(\frac{1}{25},\frac{2}{25},\frac{3}{25},\frac{4}{25},\frac{5}{25},\frac{4}{25},\frac{3}{25},\frac{2}{25},\frac{1}{25}\right).
\] 
Thus we see that $\Haus \mu\approx 0.74977.$ In this way, by choosing $p$ to be a large square, one can find self-similar measures with GCP(1) whose Hausdorff dimensions are larger but close to $1/2.$

The construction in above depends on the choice of the digits set crucially. In fact, let $p=a^2$ for a large integer $a.$ It is possible to choose a subset $A\subset\{0,\dots,a^2-1\},$ of cardinality $a+1,$ such that the cardinality of $A+A\mod a^2$ is very close to $a^2$ and $\Haus \mu_1$ is very close to one.
\section{Erd\H{o}s-Kahane arguments and spectral self-similar measures on $\mathbb{R}$} 
In this section, we provide examples of spectral self-similar measures. Intuitively, we shall see that all 'large enough' self-similar measures with 'grid structures' are spectral. Thus we will finish the proof of Theorem \ref{oneEXI}.
\subsection{Chernoff-Hoeffding estimate}\label{CH}
Let $x_1,\dots,x_n$ be an i.i.d. sequence. Furthermore, $x_1=1$ happens with probability $1-2\epsilon$ and else $x_1=0.$ We know that for large $n,$ with large probability, $\sum_i x_i\approx (1-2\epsilon)n.$ Chernoff-Hoeffding estimate provides us with a tail estimate (\cite{H63}). Let $1>\sigma,\epsilon>0,\sigma<1-2\epsilon,$ then we have
\[
Prob(\sum_i x_i\leq \sigma n)\leq e^{-n D(\sigma||1-2\epsilon)},
\]
where 
\[
D(\sigma||1-2\epsilon)=\sigma \log \frac{\sigma}{1-2\epsilon}+(1-\sigma) \log \frac{1-\sigma}{2\epsilon}.
\]
We now turn the above probability estimate to a counting estimate. Consider sequences of length $n$ over digits $\{0,\dots,p-1\}.$ Let $\sigma,\epsilon\in (0,1).$ We call a sequence $\omega$ to be $(\sigma,\epsilon,n)$-good if at least $\sigma n$ many digits are not in $[0,\epsilon p), ((1-\epsilon)p,p].$ Otherwise, we call a sequence $(\sigma,\epsilon,n)$-bad. For convenience, assume that $\epsilon p$ is an integer. Otherwise we just write $[\epsilon p]=\epsilon' p$ and replace $\epsilon$ with $\epsilon'.$ We see that the amount of $(\sigma,\epsilon,n)$-bad sequence is at most
\[
p^n e^{-n D(\sigma||1-2\epsilon)}=p^{n(1-\frac{D(\sigma||1-2\epsilon)}{\log p})}.
\]
\subsection{Almost Fourier decay for missing digits measures on $\mathbb{R}$.}\label{ErdosKahane}
Let $p> 2$ be an integer. Let $\mu$ be a ($p$-adic) missing digits self-similar measure. Let $k\in \{0,\dots p-2\}$ be the number of missing digits in $\mu,$ i.e. the IFS for $\mu$ has $p-k$ many branches. Then we see that
\[
\hat{\mu}(\xi)=\prod_{m\geq 0} \frac{1}{p-k}(\sum_j e^{-2\pi i a_jp^{-m} \xi})
\]
Let $\xi\in [p^n,p^{n+1}].$ Then we see that
\[
|\hat{\mu}(\xi)|\leq \prod_{m= 0}^{m=n} |\frac{1}{p-k}(\sum_j e^{-2\pi i a_jp^{-m} \xi})|
\]
Observe that if $\xi/p^m$ is not an integer,
\[
\sum_j e^{-2\pi i a_jp^{-m} \xi}=\frac{1-e^{-2\pi i p\xi p^{-m}}}{1-e^{-2\pi i \xi p^{-m}}}-\sum'_{j}e^{-2\pi i \xi a_j p^{-n}}.
\]
 Here $\sum'$ indicates the sum over the missing digits. The RHS is not defined for $\xi/p^m$ being integers. However, we can take the limit values for $\xi/p^m$ approaching integers and extend the domain to all $\xi.$  From here, we see that
\[
\frac{1}{p-k}|\sum_j e^{-2\pi i a_jp^{-m} \xi}|=\frac{1}{p-k}\left|\frac{1-e^{-2\pi i p\xi p^{-m}}}{1-e^{-2\pi i \xi p^{-m}}}\right|+O(\frac{k}{p-k}).
\]
The error term $O(.)$ has absolute value at most $k/(p-k).$ Now we have
\[
\left|\frac{1-e^{-2\pi i p\xi p^{-m}}}{1-e^{-2\pi i \xi p^{-m}}}\right|=\left|\frac{\sin(\pi p\xi p^{-n})}{\sin(\pi \xi p^{-n})}\right|.
\]
Let $\omega=\|\xi/p^{n}\|.$ Then we have for $1/2\geq \omega\geq 1/p,$
\[
\left|\frac{\sin(\pi p\xi p^{-n})}{\sin(\pi \xi p^{-n})}\right|\leq \frac{1}{2\omega}.
\]
Thus, suppose that the $p$-adic expansion of $|\xi|$ is a $(\sigma,\epsilon,n)$-good sequence, then
\begin{align*}
	|\hat{\mu}(\xi)|\leq \left(\frac{1}{p-k}\frac{1}{2\epsilon}+\frac{k}{p-k}\right)^{\sigma n}.\tag{Decay}
\end{align*}
Therefore we have
\[
|\log |\hat{\mu}(\xi)||/\log \xi\geq \frac{\sigma n}{n+1} \frac{|\log (1/(2\epsilon(p-k))+k/(p-k))|}{\log p}.
\]
If we keep $k$ being fixed and we choose $\epsilon$ so that 
\[
\frac{1}{2\epsilon}+k<\frac{1}{\epsilon}.
\]
Then for all large enough $n,$
\[
|\log |\hat{\mu}(\xi)||/\log \xi\geq (1+o_{n\to\infty}(1))\sigma \frac{\log (\epsilon (p-k))}{\log p}.
\]
The number of $\xi\in [p^n,p^{n+1})$ with $(\sigma,\epsilon,n)$-bad expansion is at most
\[
p^{n(1-\frac{D(\sigma||1-2\epsilon)}{\log p})}.
\]
We have therefore proved the following result.
\begin{thm}[Erd\H{o}s-Kahane]\label{EK}
	Let $k\geq 1$ be an integer. Let $\mu$ be a $p$-adic missing digits measure with $k$ many missing digits. There is a pair of numbers $\lambda,\Delta\in (0,1)$ such that for $N$ being large enough. For all but except at most $N^{1-\lambda}$ many integers $\xi\in [1,N]$ we have
	\[
	|\hat{\mu}(\xi)|\leq \xi^{-\Delta}.
	\]
	Moreover, $\lambda, \delta$ can be determined explicitly, in particular, if $k$ is fixed and $p$ is large enough, we can choose $\sigma,\epsilon\in (0,1),$ $1\leq\epsilon p\leq p-1$ and $2\epsilon k<1,$
	\[
	\lambda=\frac{D(\sigma||1-2\epsilon)}{\log p}
	\]
	and 
	\[
	\Delta=\sigma \frac{\log (\epsilon (p-k))}{\log p}.
	\]
\end{thm}
Of course, for applications, we want $\lambda,\Delta$ to be both as large as possible. It is in particular interesting to achieve $\Delta>1/2$ and $\lambda+\Haus \mu>1.$ at first glance, this seems to be impossible as we want to achieve a Fourier power decay that corresponds to a Hausdorff dimension greater than one. However, we are not asking for a decay $|\hat{\mu}(\xi)|=O(|\xi|^{-\Delta}).$ We are asking for a good decay to hold on 'most of' the frequencies. For our special self-similar measures, this can be achieved. 
\subsection{Bonus observation}\label{Bonus}
The above arguments  in fact lead us to a slightly stronger result. We have only considered $\xi$ to be integers. However, the above arguments in fact hold for $\xi$ being real numbers (Fourier transform) and the Erd\H{o}s-Kahane argument applies to their integer parts (which would already have enough digits to conclude the decay). We have met this type of arguments in Section \ref{Missing Digits Measures}.
\subsection{A simple construction}\label{Simple}
Let $k\geq 1$ be an integer. Let $p>k$ be an integer. We construct $p$-adic missing digits measures with only $k$ missing digits. It is straightforward to see that
\[
\dim_{l^2}\mu=\Haus \mu=\frac{\log(p-k)}{\log p}
\]
and $\mu$ is AD-regular. To be concrete, we let $k<50.$ We choose $\epsilon=0.01$ and $\sigma=0.97.$ Then $D(\sigma||1-2\epsilon)>0.$ We see that for all large enough $p,$
\[
\Delta=\sigma \frac{\log (\epsilon (p-k))}{\log p}>1/2,
\]
and
\begin{align*}
	\Haus \mu+\lambda=&\frac{\log (p-k)}{\log p}+\frac{D(\sigma||1-2\epsilon)}{\log p}&\\
	=&1-\frac{k}{p\log p}+\frac{D(\sigma||1-2\epsilon)}{\log p}+O(p^{-2})>1&.
\end{align*}
Here, we have used the fact that
\[
\log (1-kp^{-1})=1-kp^{-1}+O(p^{-2}).
\]
Thus, as long as $k<50$ is fixed, $p$-adic missing digits measures (with equal weights) are spectral and thick measures for all large enough $p.$ For a more concrete example. We fix $k=1.$ The it is possible to see that as long as $p\geq 13417,$ $p$-adic measures with one missing digit is spectral and thick. Later on, we will make some numerical arguments to have some smaller values for $p.$
\subsection{General scaling ratios}\label{General scaling ratios}
\subsubsection{Many branches}
Here we  find spectral and thick self-similar measures whose scaling ratio may not be rational numbers. Now we deal with the case when the scaling ratio of $\mu$ is $a\in (1/p,1/(p-1))$ and translations are integers in $\{0,\dots,p-1\}.$ We fix an integer $k,$ and we make the self-similar measure to have $p-k $ many branches with equal weight. We want to show that fas long as $p$ is large enough, the self-similar measure is spectral and thick.  Most of the arguments in the previous section apply for this setting. However, we can no longer use the digit expansion trick to obtain good Fourier decay. Instead, we will make the following more general counting argument.

Let $a^{-1}\in (p,p+1).$ Let $n$ be an integer. Consider the interval $[a^{-n},a^{-n-1}].$ For Fourier decay, we need to consider
\[
\left|\frac{\sin(\pi a^{-1}\omega a^{m})}{\sin(\pi \omega a^{m})}\right|.
\]
For this, we need to consider $\|\omega a^m\|$ for $m=0,1,\dots,n.$ Consider the following set of integers
\[
A(\sigma,\epsilon,a,n)=\{\omega\in [a^{-n},a^{-n-1}]\cap\mathbb{N}: \|\omega a^m\|>\epsilon \text{ for at most } \sigma n \text{ many } m\in\{0,\dots n\}\}.
\]

We now want to estimate the size of the above set. First, out of the set $\{1,\dots,n\}$ we choose a subset $\mathcal{N}$ of cardinality at most $\sigma n.$ Next, observe that $\omega a^{n}\in [1,a^{-1}].$ 

If $n\notin \mathcal{N},$ then we require that $\|\omega a^n\|\leq\epsilon.$ There are at most $[a^{-1}]+1$ many integers in this range. Thus there are at most $[a^{-1}]+1$ many $\epsilon$-balls for $\omega a^{n}$ to be contained in. 

If $n\in \mathcal{N},$ we then require that $\|\omega a^{n}\|>\epsilon.$ Then we see that there are at most $[a^{-1}]+1$ many intervals of length $1-2\epsilon$ (or $1/2-\epsilon$-balls) for $\omega a^n$ to be contained in. 

Next, we consider possible choices for $\omega a^{n-1}.$ If $n-1\notin\mathcal{N},$ then we require that $\|\omega a^{n-1}\|\leq \epsilon.$ Suppose that $n\notin\mathcal{N},$ then the possible choices for $\omega a^{n-1}$ must be contained in a disjoint union of at most $[a^{-1}]+1$ many $a^{-1}\epsilon$-balls. Furthermore, inside each such ball, there are at most $[a^{-1}\epsilon]+1$ many integers. Thus, the possible range for $\omega a^{n-1}$ is then a disjoint union of at most $([a^{-1}]+1)([a^{-1}(2\epsilon)]+1)$ many $\epsilon$-balls. Suppose that $n\in\mathcal{N},$ then similarly, we see that the possible range for $\omega a^{n-1}$ is a disjoint union of at most $([a^{-1}]+1)([a^{-1}(1-2\epsilon)]+1)$ many $\epsilon$-balls. Similar argument can be done in the case when $n-1\in\mathcal{N}.$

We can perform the above steps for each $m\in \{n,\dots,1\}.$ As a result, we see that in the end, there are at most
\[
([a^{-1}]+1)([a^{-1}(1-2\epsilon)]+1)^{\#\mathcal{N}}([a^{-1}(2\epsilon)]+1)^{n-\#\mathcal{N}}
\]
choices for $\omega.$  Therefore, we see that
\[
\# A(\sigma,\epsilon,a,n)\leq ([a^{-1}+1])\sum_{\mathcal{N}} ([a^{-1}(1-2\epsilon)]+1)^{\#\mathcal{N}}([a^{-1}(2\epsilon)]+1)^{n-\#\mathcal{N}},
\]
where the sum is ranging over all possible subset $\mathcal{N}\subset \{1,\dots,n\}$ with $\#\mathcal{N}\leq \sigma n.$ We can arrange the sum according to the cardinality of $\mathcal{N},$
\[
\# A(\sigma,\epsilon,a,n)\leq ([a^{-1}+1])\sum_{1\leq k\leq \sigma n}\binom{n}{k} ([a^{-1}(1-2\epsilon)]+1)^{k}([a^{-1}(2\epsilon)]+1)^{n-k}.
\]
From this and the Stirling's estimate for binomial coefficients, we see that
\[
\limsup_{n\to\infty} \frac{\log \#A(\sigma,\epsilon,a,n)}{n\log a^{-1}}\leq \frac{H(\sigma)}{-\log a}+\sigma \frac{\log ([a^{-1}(1-2\epsilon)]+1)}{-\log a}+(1-\sigma) \frac{\log ([a^{-1}(2\epsilon)]+1)}{-\log a},
\]
where $H(\sigma)=-\sigma\log \sigma-(1-\sigma)
\log (1-\sigma)$. We can rearrange the above as
\[
\limsup_{n\to\infty} \frac{\log \#A(\sigma,\epsilon,a,n)}{n\log a^{-1}}\leq \frac{DD(\sigma,a,\epsilon)}{-\log a},
\] 
where
\[
DD(\sigma,a,\epsilon)=\frac{H(\sigma)+\sigma \log ([a^{-1}(1-2\epsilon)]+1)+(1-\sigma) \log ([a^{-1}(2\epsilon)]+1)}{-\log a}.
\]
The purpose of this step is to compare the above with the formula for $\lambda$ we had in Theorem \ref{EK}. We want to use $DD(\sigma,a\epsilon)$ in the place of $D(\sigma||1-2\epsilon)$ and our new $\lambda$ is
\[
\lambda=\frac{DD(\sigma,a,\epsilon)}{-\log a}.
\] 
This will make the choice of $\lambda$ to be a bit worse, i.e. smaller than the case for $a=1/p.$  However, if $p$ is large enough, then $a$ is small enough and the defect on $\lambda$ tends to zero. The rest of the argument will be similar to what we already did for the case when $a^{-1}=p$ is an integer. Thus as long as $p$ is large enough $\mu$ is thick and spectral. 
\subsubsection{A question for two branches}
If we let $p=2,$ then the self-similar measure we constructed above has two branches. If we choose $a$ to be very close to but smaller than $1/2,$ then this self-similar measure will have Hausdorff dimension close to one (more precisely, $\log 2/\log (a^{-1})$).  We are not able to estimate its $l^1$-dimension. However, we suspect that as long as $a$ is sufficiently close to $1/2$, then $\dim_{l^1}\mu>1/2.$
\subsection{Numerical computations}\label{Simple2}
In this section, we will perform some numerical computations to find some more spectral and thick measures. The reader can skip this section as it will not contribute to the main theory other than reducing the value $p=13417$ to $p=15$ for the middle-$p$th Cantor measure to be spectral and thick and the fact that the middle-third Cantor measure is spectral. Better numerics are likely to exist. Numerical methods we introduce here can be extended to deal with more general self-similar measures other than missing digits measures (by making the arguments in Section \ref{General scaling ratios} to be more quantitative).  However, this is not the main scope of this paper and we therefore hope to revisit this matter in the future.

\subsubsection{The Lyapunov function and the first discretization method}
Let $p\geq 3$ be an integer. Let $0=a_1<a_2<\dots<a_{r}\leq p-1$ be $1<r<p$ integers. We then construct the missing digits measure $\mu.$ Let $\phi$ be the following function
\[
\phi(x)=-\frac{\log |r^{-1}\sum_{i=1}^r  e^{-2\pi i a_i x}|}{\log p}.
\]
We call this function the \emph{Lyapunov function} of $\mu.$ This function is $\mathbb{Z}$-periodic on $\mathbb{R}.$ Moreover, it is a positive and locally integrable function. Now, we construct the following $1/p$-discretized function
\[
\phi_p(x)=\inf_{y\in I^p_x}\phi(y),
\]
where $I^p_x$ is the interval of form $[j/p,(j+1)/p),j\in\mathbb{Z}.$ Now $\phi_p$ is a step function. We only to consider $\phi_p$ on $[0,1].$ Consider the random variable $X$ which takes the values of $\phi_p$ with equal probability. There are $p$ possible values for $\phi_p$ counting multiplicities. Thus $X$ is a Bernoulli variable with $p$ values or more intuitively, a fair coin with $p$-sides.

Let $m_\mu=E[X]$, $\sigma^2_\mu=Var[X],$ $C_\mu=\max X.$ Let $n\geq 1$ be an integer and let $X_1,\dots,X_n$ be i.i.d random variables distributed as $X.$ Let $S_n=\sum_{i=1}^n X_i.$ Then by Bernstein's inequality \cite{Bern}, we see that
\[
Prob(|S_n-n m_\mu|>t)\leq 2 \exp \left(-\frac{t^2/2}{n \sigma^2_\mu+3^{-1}C_\mu t}\right).
\]
for all $t>0.$ We can apply the arguments in Sections \ref{CH}, \ref{ErdosKahane} to see that as long as $m_\mu>1/2,$ $\mu$ is spectral with
\[
\dim_{R}\mu\geq \frac{1}{\log p}\frac{2^{-1}(m_\mu-2^{-1})^2}{\sigma^2_\mu+3^{-1}C_\mu (m_\mu-2^{-1})}.
\]
We know that $\Haus \mu=\dim_{l^2}\mu=\log r/\log p.$ We provide some examples.
\begin{exm}
	Let $p=6,r=1.$ Let $\{0,1,2,3,4\}$ be the digits set, i.e. we have the missing digit $5.$ Then it is possible to see that
	\[
	m_\mu\approx 0.557317>1/2.
	\]
	Thus, the base $6$-missing digits measure with missing digit $5$ is spectral.
\end{exm}

\begin{exm}
	Let $p=12,r=1.$ Let $\{0,1,2,3,4,5,6,7,8,9,10\}$ be the digits set, i.e. we have the missing digit $11.$ Then it is possible to see that
	\[
	m_\mu\approx 0.700569 >1/2
	\]
	and
	\[
	\Haus\mu+\dim_R\mu\geq 1.0081>1.
	\]
	Thus, the base $12$-missing digits measure with missing digit $11$ is spectral and thick.
\end{exm}

\begin{exm}\label{15th}
	Let $p=15,r=1.$ Let $\mu$ be the middle-$15$th Cantor measure. Then it is possible to see that
	\[
	m_\mu\approx 0.67345 >1/2
	\]
	and
	\[
	\Haus\mu+\dim_R\mu\geq 1.00756 >1.
	\]
	Thus, the middle-$15$th Cantor measure is spectral and thick.
\end{exm}
\subsubsection{Higher order discretization methods}
Now we continue going along this route. Let $k\geq 1$ be an integer. We consider the $1/p^k$-discretized function
\[
\phi_{p^k}(x)=\inf_{y\in I^{p^k}_x}\phi(y).
\]
As before, let $X$ be the random variable taking the same values of $\phi_{p^k}$, each with probability $1/p^k.$ Again, We define
\begin{align*} 
	&m_\mu=E[X],\\ &\sigma_\mu^2=Var[X],\\ &C_\mu=\max X.
\end{align*}
Notice that they also depend on $k.$ We will write, for example, $m_\mu(k)$ when it is necessary to display the dependence on $k$.

Let $n\geq k$ be an integer. Let $\xi$ be an integer in $[p^{n}, p^{n+1}].$ Then it has $n$ digits in base $p$ expansion. We can divide these $n$ digits into blocks of $k$ digits and view this as a base $p^k$ expansion. More precisely, let $j\in\{0,\dots,k-1\}.$ We can expand $[\xi/p^j]$ in base $p^k,$ i.e.
\[
[\xi/p^j]=a_{j,n_j}p^{n_j k}+\dots +a_{j,1}.
\] 
In particular, we see that
\[
[(n-k)/k]\leq n_j\leq [n/k]+1.
\]
For each fixed $j\in\{0,\dots,k-1\},$ we can consider a sequence of $n_j$ many i.i.d random variables $X_1,\dots,X_{n_j}$ distributed as $X.$ Then we see that for $t_j>0,$
\begin{align*}
	|X_1+\dots+X_{n_j}-n_j m_\mu|>t_j\tag{G}
\end{align*}
with probability at most
\[
2 \exp \left(-\frac{t_j^2/2}{n_j\sigma^2_\mu+3^{-1}C_\mu t_j}\right).
\]
Thus the number of possible sequences to realize (G)  is at most
\[
2p^{n_jk}\exp \left(-\frac{t_j^2/2}{n_j \sigma^2_\mu+3^{-1}C_\mu t_j}\right).
\]
We can perform the above steps $k$ times, one for each $j.$ Suppose that the base $p^k$ expansions of $[\xi/p^j]$ all satisfy (G), then we would have
\[
-\log\hat{\mu}(\xi)/\log (p^n)\geq -\frac{t_0+\dots+t_{k-1}}{n}+m_\mu.
\]
Suppose that $m_\mu>1/2.$ We choose $t_j=n_j(m_\mu-2^{-1})$ and see that
\[
\frac{t_0+\dots+t_{k-1}}{n}+m_\mu=\frac{1}{2}.
\]
Thus we see that
\[
-\log\hat{\mu}(\xi)/\log (p^n)\geq \frac{1}{2}
\]
for all but except at most
\[
\sum_{j} 2p^{n_jk}\exp \left(-\frac{t_j^2/2}{n_j \sigma^2_\mu+3^{-1}C_\mu t_j}\right)
\]
many $\xi\in [p^n,p^{n+1}].$ If $n$ is large enough, the above value is
\[
\ll p^n \exp\left(-\frac{n(m_\mu-0.5)^2/2}{k\sigma_\mu^2+3^{-1}kC_\mu (m_\mu-0.5)}  \right).
\]
Thus we see that $\mu$ is spectral and
\[
\dim_R \mu\geq \frac{1}{\log p}\frac{(m_\mu-0.5)^2/2}{k\sigma_\mu^2+3^{-1}kC_\mu (m_\mu-0.5)}.
\]
It is possible to see that as $k\to\infty,$ the values
\[
m_{\mu}(k)\to\int_{[0,1]} \phi(x)dx=m,
\]
\[
\sigma^2_\mu(k)\to \int_{[0,1]} (\phi(x)-m)^2dx=\sigma^2.
\]
The value $C_\mu(k)$ may tend to $\infty.$ If $k$ is large, then the above estimate for $\dim_R \mu$ is likely to be too small to be useful. On the other hand, we do have a better sufficient condition for the spectral property.
\begin{thm}\label{Spectral Measures}
	Let $p\geq 3$ be an integer. Let $\mu$ be a $p$-adic missing digits measure on $\mathbb{R}.$ Let $\phi$ be its Lyapunov function. If 
	\[
	\int_{[0,1]}\phi(x)dx>1/2
	\]
	then $\mu$ is spectral. Let $k\geq 1$ be the smallest integer such that $m_{\mu}(k)>1/2.$ Then we have
	\[
	\dim_R \mu\geq \frac{1}{\log p}\frac{(m_\mu(k)-0.5)^2/2}{k\sigma_\mu^2(k)+3^{-1}kC_\mu(k) (m_\mu(k)-0.5)}.
	\]
	Here $m_\mu(k),\sigma_\mu(k),C_\mu(k)$ are as defined in the beginning of this subsection.
\end{thm}
\begin{exm}
	For the middle-third Cantor measure $\mu,$ it is possible to see that
	\[
	\int_{[0,1]}\phi(x)dx=\frac{\log 2}{\log 3}>\frac{1}{2}.
	\]
	Therefore $\mu$ is spectral. In fact, by taking $k=6$ in above, we see that
	\[
	m_\mu(6)\approx 0.614731
	\]
	and
	\[
	\dim_R\mu\geq 0.0012797.
	\]
	This is very far from being able to conclude that $\mu$ is thick which requires $\dim_R \mu$ to be at least $1-\log 2/\log 3\approx 0.369.$
\end{exm}
It is in general not true that
\[
\int_{[0,1]}\phi(x)dx=\Haus \mu.
\]
However, this is the case for $p$ being an odd integer and for the middle-$p$-th Cantor measures. This can be proved with the help of the following formula
\[
\int_{[0,1/2]}\log (\sin(2\pi x))dx=-\frac{\log 2}{2}.
\]

\section{Lattice counting for self-similar measures on $\mathbb{R}^n, n\geq 2$}\label{ND}
In this section, we consider self-similar measures (missing digits measures) on higher dimensional Euclidean spaces. The situation becomes a little bit more subtle than the one dimensional case. For the higher dimensional cases, we will only provide examples rather than a general theory. We will prove Theorem \ref{NEXI}.

The measures we will consider in this section are missing digits measures, thus their Hausdorff dimensions and $l^2$-dimensions are equal. We note here that our method can deal with more general self-similar measures, for example, the contraction ratio does not need to be a rational number and the probability vector does not need to be equally weighted.
\subsection{Basic set up}
Let $p>1$ be an integer and let $r\in\{2,\dots,p^n-1\}.$ We choose $r$ translations vectors out of $\mathbb{Z}^n\cap [0,p-1]^n.$ Let $\mu$ be the self-similar measures constructed with the above data and with the probability vector which gives each branch the weight $1/r.$ In this case, we have $\dim_{l^2}\mu=\Haus \mu.$

First, we recall the Fourier transform formula:
\[
\hat{\mu}(\xi)=\prod_{m\geq 0} \frac{1}{p^n-k}(\sum_j e^{-2\pi i p^{-m} (a_j,\xi)}),
\]
where $k=p^n-r$ is the number of missing digits and $(,)$ is the standard Euclidean  inner product. Now we have
\[
\frac{1}{p^n-k}\sum_j e^{-2\pi i p^{-m} (a_j,\xi)}=\frac{1}{p^n-k}\prod_{s=1}^n \frac{1-e^{-2\pi i p\xi_s/p^m}}{1-e^{-2\pi i \xi_s/p^m}}+O(\frac{k}{p^n-k}).
\]
In fact, the $O(.)$ term has absolute value at most
\[
\frac{k}{p^n-k}.
\]
We can apply Theorem \ref{EK}. We need to choose $\Delta>n/(n+1).$ For now, we do not restrict $\lambda.$ The conditions for it will be given later. Now, let $\epsilon>0$ be a small number. Let $t$ be a large integer. Suppose that at least one of the coordinates of $\xi$ is $(\Delta,\epsilon, t)$-good, then we have
\[
\prod_{m\geq 0}\frac{1}{p^n-k}\left(\left|\prod_{s=1}^n \frac{1-e^{-2\pi i p\xi_s/p^m}}{1-e^{-2\pi i \xi_s/p^m}}\right|+k\right)\leq \left(\frac{1}{p^n-k}\left(\frac{p^{n-1}}{2\epsilon}+k\right)\right)^{\Delta t}.
\]
Suppose that $k<p^{n-1}/(2\epsilon).$  Since $\Delta>n/(n+1),$ we see that
\begin{align*}
	-\log |\hat{\mu}({\xi})|/\log (p^t)>\frac{n}{n+1} \frac{|\log p^{n-1}/(\epsilon (p^n-k))|}{\log p}\tag{D}
\end{align*}
as long as one of the coordinates of $\xi$ is $(\Delta,\epsilon,t)$-good. In particular if $k,\epsilon$ are fixed then as long as $p$ is large enough, (D) implies that for a constant $\Delta'>n/(n+1),$
\[
|\hat{\mu}(\xi)|\leq |\xi|^{-\Delta'}.
\]
We define the following sets
\[
GOOD_t=\left\{x\in [0,p^t]: x \text{ is } (\Delta,\epsilon,t)\text{-good}\right\},
\]
\[
BAD_t=[0,p^t]\setminus GOOD_t.
\]
For all large enough $t,$ we have
\begin{align*}
	\#BAD_t\ll p^{t(1-\lambda)}.\tag{B}
\end{align*}
Let $x\geq 0$ be an integer, we say that $x$ is good if it is in $GOOD_t$ for some $t>0,$ else, we say that $x$ is bad.
\subsection{Lattice counting estimates} 
In this and the next sections, we want to find self-similar measures $\mu$ such that for a constant $c>0,$
\[
\sum_{q=Q}^{2Q}\mu(A(\delta,q,\gamma))\asymp Q\delta^n
\]
for $\delta\gg Q^{-(1/n)-c}.$ Let $Q,\delta$ be given. As before, we want to estimate
\[
\sum_{q=Q}^{2Q}\sum_{\xi\neq 0, q|\xi, |\xi|\leq Kq/\delta} |\hat{\mu}(\xi)|.
\]
From the above we see that
\[
\sum_{\xi\neq 0,q|\xi, |\xi|\leq Kq/\delta} |\hat{\mu}(\xi)|\leq \sum_{\xi\neq 0,q|\xi, |\xi|\leq Kq/\delta} |\xi|^{-\Delta'}+\sum_{\{\xi:q|\xi, |\xi|\leq Kq/\delta, \text{ all coordinates are bad}\}} |\hat{\mu}(\xi)|.
\]
For convenience, we write
\[
B_q=\{\xi:q|\xi, |\xi|\leq Kq/\delta, \text{ all coordinates are bad}\}=\{x\in [0,Kq/\delta], q|x, x\text{ is bad}\}^n,
\]
where the exponent on the rightmost expression denotes the Cartesian product. Again, for convenience, we write
\[
\#_q=\#B_q.
\]
Then we see that (we again omit a divisor function as a multiplicative factor)
\[
\sum_{q=Q}^{2Q}\sum_{\xi\neq 0,q|\xi, |\xi|\leq Kq/\delta} |\hat{\mu}(\xi)|\leq (\sum_q\sum_{\xi\neq 0, q|\xi,|\xi|\leq 2KQ/\delta} |\xi|^{-\Delta'})+(\sum_q\sum_{B_q} |\hat{\mu}(\xi)|).
\]
The first term on the RHS in above is
\[
\ll (KQ/\delta)^{-\Delta'}\times \frac{QK^n}{\delta^n}.
\]
For the second term, we use Cauchy-Schwarz and ($L_2$) in Section \ref{AD},
\[
\sum_q\sum_{B_q} |\hat{\mu}(\xi)|\leq \sqrt{\sum_q \#_q} (Q/\delta)^{(n-\Haus \mu)/2}.
\]
Now we see that
\[
\sum_q \#_q=\sum_q (\#\{x\in [0,Kq/\delta], q|x, x\text{ is bad}\})^n
\]
We know that $$\sum_{q=Q}^{2Q} \#\{x\in [0,Kq/\delta], q|x, x\text{ is bad}\}\leq (2KQ/\delta)^{1-\lambda}$$
and 
$$\#\{x\in [0,Kq/\delta], q|x, x\text{ is bad}\}\leq K/\delta.$$
Thus, 
$$\sum_q (\#\{x\in [0,Kq/\delta], q|x, x\text{ is bad}\})^n$$ is maximized when $\#\{x\in [0,Kq/\delta], q|x, x\text{ is bad}\}$ all takes the largest possible value (to achieve the maximal degree of non-uniformity).\footnote{More precisely, let $a_1,\dots,a_Q$ are positive numbers. In order to maximize the sum $\sum_{i=1}^{Q} a^n_i$ under the condition that $\sum_{i=1}^Q a_i=S_1$, a fixed positive value and $\max_{i} a_i\leq S_2$, another fixed positive value, we need to choose $a_1,\dots,a_Q$ in such a way that they values are either zero or $\min\{S_1,S_2\}$ except for at most one term. If $S_2<S_1,$ then one possible solution is $\{S_2,\dots,S_2,0,\dots,0,S'\}$ so that there are $[S_1/S_2]$ many $S_2's$ and $S'=S_1-S_2[S_1/S_2].$} From here, we have
\[
\sum_{q=Q}^{2Q} (\#\{x\in [0,Kq/\delta], q|x, x\text{ is bad}\})^n\ll ((KQ/\delta)^{1-\lambda}\times \frac{1}{K/\delta}+1) \times (K/\delta)^n.
\]
Thus we see that ($a,b>0,\sqrt{a+b}\leq \sqrt{a}+\sqrt{b}$),
\begin{align*}
	\sum_{q=Q}^{2Q}\sum_{\xi\neq 0,q|\xi, |\xi|\leq Kq/\delta} |\hat{\mu}(\xi)|\ll (KQ/\delta)^{-\Delta}\times \frac{QK^n}{\delta^n}\\+(KQ/\delta)^{(1-\lambda+n-\Haus \mu)/2}(K/\delta)^{(n-1)/2}+(K/\delta)^{n/2}(KQ/\delta)^{(n-\Haus\mu)/2}.
\end{align*}
We need the above to be $\ll Q$ for a good counting. Again, as before, we can choose $K=(1/\delta)^{\epsilon}$ for an arbitrarily small positive number $\epsilon.$ Thus, by choosing $\epsilon>0$ to be small enough we see that the above is $\ll Q$ if for an $\epsilon'>0$
\[
\delta\gg Q^{\epsilon'-(\Delta/(n-\Delta))},
\]
\[
\delta\gg Q^{\epsilon'+(1-\lambda+n-\Haus\mu-2)/(n-1+1-\lambda+n-\Haus\mu)},
\]
\[
\delta\gg Q^{\epsilon'+(n-\Haus \mu-2)/(2n-\Haus \mu)}.
\]
We want that the exponents on $Q$ in above are all smaller than $-1/n.$ This is achieved by the following conditions (we suppressed $\epsilon'$ into the strict inequality signs $'>'$):
\begin{align*}
	\tag{C}\Delta'>n/(n+1),\\
	\Haus \mu>n-\lambda,\\
	\Haus \mu>n\frac{n}{n+1}.
\end{align*}
The first condition is satisfied as we already required that $\Delta'>n/(n+1).$ The rest two conditions say that $\Haus \mu$ should be sufficiently close to $n.$ From here, under conditions (C), we see that for a constant $c>0,$
\[
\sum_{q=Q}^{2Q}\mu(A(\delta,q,\gamma))\asymp Q\delta^n
\]
for $\delta\gg Q^{-(1/n)-c}.$
\subsection{Examples}
We can now choose $\Delta=(n+1)/(n+2)>n/(n+1).$ We also choose $\epsilon$ to be any number smaller than $(1-\Delta)/2>0.$ Then we have $D(\Delta||1-2\epsilon)>0.$ From Theorem \ref{EK} we see that the value of $\lambda$ in (B) can be chosen to be
\[
\lambda=\frac{D(\Delta||1-2\epsilon)}{\log p}.
\]
Let $\mu$ be a missing digits measure with $k$ missing digits. Then we have
\[
\Haus \mu=\frac{\log (p^n-k)}{\log p}.
\]
We need to make Conditions (C) valid. Observe that as long as $p$ is large enough and $k\geq 1$ is fixed, the first and third conditions are satisfied. For the second condition, we see that
\[
\frac{\log (p^n-k)}{\log p}+\frac{D(\Delta||1-2\epsilon)}{\log p}=n-\frac{k}{p^n \log p}+\frac{D(\Delta||1-2\epsilon)}{\log p}+O(1/p^{2n})>n
\]
as long as $p$ is large enough.

\section{Lattice counting and metric Diophantine approximation}\label{Dio}
We have finished the latter counting consideration in this paper. Now in this and the next Sections, we consider some applications of the lattice counting estimates in metric Diophantine approximation. Most of the arguments are straightforward but complicated. However, most of the ideas are not new. In fact, the proofs of Theorems \ref{Good Counting to Metric},\ref{Jarnik1} are  standard arguments using the convergence Borel-Cantelli lemma (and the Hausdorff-Cantelli lemma); the proof of Theorem \ref{Jarnik2} uses a similar argument in \cite{BV}; the proof of Theorem \ref{Jarnik4} follows closely \cite{B34}; finally, the proofs of the results in Section \ref{divergence} utilize some arguments in \cite[Section 4]{ACY}.

\subsection{A Khinchine convergence result}
Lattice counting estimates are very useful for metric Diophantine approximations. First, we prove the following standard result.

\begin{thm}\label{Good Counting to Metric}
	Let $n\geq 1$ be an integer. Let $\gamma\in [0,1]^n.$ Let $\mu$ be a Borel probability measure on $[0,1]^n$ with GCP(n), i.e.  for large enough $Q,$
	\[
	\sum_{q=Q}^{2Q}\mu(A(\delta,q,\gamma))\ll Q\delta^n.
	\]
	for $1>\delta\gg Q^{-(1/n)-c}$ where $c>0$ is a constant. Then $\mu(W(\psi,\gamma))=0$ as long as $\sum_{q\geq 1} \psi(q)^n<\infty.$
\end{thm}
\begin{proof}
	Let $\psi:\mathbb{N}\to (0,1/2)$ be a monotonic approximation function with
	\[
	\sum_{q} (\psi(q))^n<\infty.
	\]
	Without loss of generality, we can assume that $\psi(q)\leq q^{-1/n}.$ Indeed, if $\psi(q)\geq q^{-1/n}$ for some $q\geq 1000,$ then $\psi(q')\geq q^{-1/n}$ for $q'\leq q.$ Thus we have
	\[
	\sum_{q'\in [[q/2],q]}(\psi(q'))^n\geq \frac{1}{2}.
	\]
	Thus if $\psi(q)\geq q^{-1/n}$ for infinitely many $q,$ then 
	\[
	\sum_q (\psi(q))^n=\infty.
	\]
	Thus $\psi(q)\leq q^{-1/n}$ for all but except at most finitely many $q.$ We can then assume that $\psi(q)\leq q^{-1/n}$ for all $q.$ By assumption, for all large $Q$, for $\delta\gg Q^{-1/n-c,}$ we have
	\[
	\sum_{q=Q}^{2Q}\mu(A(\delta,q,\gamma))\asymp Q\delta^n.
	\]
	For each $Q,$ we can choose $\delta_Q=\max_{q\in [Q,2Q]}\{ \psi(q)\}.$ Two possibilities can happen,  either $\delta_Q\geq Q^{-1/n-c}$ or $\delta_Q\leq Q^{-1/n-c}.$ For the first case, we have
	\[
	\sum_{q=Q}^{2Q}\mu(A(\delta_Q,q,\gamma))\asymp Q\delta_Q^n.
	\]
	For the second case, we have
	\[
	\sum_{q=Q}^{2Q}\mu(A(\delta_Q,q,\gamma))\leq \sum_{q=Q}^{2Q}\mu(A(Q^{-1/n-c},q,\gamma))\asymp Q \frac{1}{Q^{1+nc}}=\frac{1}{Q^{nc}}.
	\]
	Collecting the above results, we see that
	\[
	\sum_{q\geq 1}\mu(A(\psi(q),q,\gamma))\leq \sum_{k\geq 0} \sum_{q=2^k}^{2^{k+1}-1} \mu(A(\delta_{2^k},q,\gamma))\ll \sum_{k\geq 0} \max\{2^k\delta^n_{2^k},2^{-nck}\}.
	\]
	Observe that
	\[
	\sum_{k\geq 0} 2^{-nck}<\infty,
	\]
	and
	\[
	\sum_{k\geq 0} 2^k \delta^n_{2^k}\leq 1+\sum_{k\geq 1} 2^k \frac{\sum_{q=2^{k-1}+1}^{2^k} (\psi(q))^n}{2^{k}-2^{k-1}}\ll \sum_{q\geq 1} (\psi(q))^n<\infty.
	\]
	Thus $\sum_q \mu(A(\psi(q),q,\gamma))<\infty.$ By the convergence Borel-Cantelli lemma, we conclude the result.
\end{proof}
\subsection{Besicovitch-Jarnik type results}
In this section, we will discuss some Jarnik type results.  Combining all the results in this section will conclude the proof of Theorem \ref{Main}(1)(3) as well as Theorem \ref{MainInhomo}(2).
\begin{thm}\label{Jarnik1}
	Let $n\geq 1$ be an integer. Let $s\in (0,n).$ Let $\mu$ have GCP(n) and be $s$-dimensional AD-regular. Then there is a $c>0$ such that for $c'\in (0,c),$
	\[
	\Haus W(\psi,\gamma)\cap supp(\mu)\leq s-\frac{nc'}{\frac{1}{n}+c'+1},
	\]
	for the approximation function $\psi(q)=q^{-(1/n)-c'}$ and any $\gamma\in [0,1]^n.$
\end{thm}
\begin{proof}
	As $\mu$ is $s$ dimensional AD-regular for some $s\in (0,n),$  it is possible to see that
	\[
	\#\{\mathbf{k}\in\mathbb{Z}^n: d((\mathbf{k}-\gamma)/Q, supp(\mu)) \leq \delta/2   \}\ll \mu(A(\delta,Q,\gamma))/(\delta/Q)^{s},
	\]
	where $d(.,.)$ denotes the standard Euclidean distance. Let $s'\in (0,n).$ Consider the $\delta/Q$-neighbourhood $N_\delta$ of 
	\[
	\{(\mathbf{k}-\gamma)/Q: \mathbf{k}\in\mathbb{Z}^n, d((\mathbf{k}-\gamma)/Q, supp(\mu))<\delta/2\}. 
	\]
	Assume that $\delta<1/2$ so that $N_\delta$ is a union of disjoint $\delta/Q$-balls. Let $s'>0.$ The we see that the sum of $|B|^{s'}$ with $B$ ranging over the balls forming $N_\delta$ is
	\begin{align*}
		\ll \frac{\mu(A(\delta,Q,\gamma))}{(\delta/Q)^s}\left(\frac{\delta}{Q}\right)^{s'}\ll \mu(A(\delta,Q,\gamma))(\delta/Q)^{s'-s}.\tag{+}
	\end{align*}
	Now, GCP(n) for $\mu$ provides us with a number $c>0$ such that
	\[
	\sum_{q=Q}^{2Q}\mu(A(\delta,q,\gamma))\ll Q\delta^n
	\]
	for $\delta\gg Q^{-(1/n)-c}.$ Let $c'\in (0,c].$ Combine this with the estimate (+) and apply the Hausdorff-Cantelli lemma (see Lemma \ref{HC}), we see that $\mathcal{H}^{s'}(W(\psi,\gamma)\cap supp(\mu))=0$ for $\psi(q)=q^{-(1/n)-c'}$ and 
	\[
	s'> \max\left\{s-\frac{nc'}{\frac{1}{n}+c'+1},0\right\}.
	\]
	Usually, $c$ is a very small number so the RHS in above is usually bigger than zero and slightly smaller than $s.$ This implies that
	\[
	\Haus W(\psi,\gamma)\cap supp(\mu)\leq \max\left\{s-\frac{nc'}{\frac{1}{n}+c'+1},0\right\}.
	\]
	This finishes the proof.
\end{proof}
Now, we make use of the lower bound of GCP(n). Here, we only consider the case when $n=1$ and $\gamma=0.$ We prove the following results. The proof we found is inspired by the argument in \cite{BV} for the proof of the mass transference principle. In what follows, let $\mu$ be a missing digits measure which is spectral and thick. Let $\mu'$ be a branch of $\mu.$ Then $\mu'$ is a scaled and translated copy of $\mu.$ Many properties of $\mu$ are also valid for $\mu'.$ For example, if $\mu$ is spectral and thick, we see that $\mu'$ is spectral and thick as well. Of course, their implied constants in Definition \ref{DST} are not the same but the $l^2$-dimension and the residue dimension will keep unchanged. 
\begin{thm}\label{Jarnik2}
	Let $t\geq 1$ be an integer. Let $p\geq 3$ be an integer. Let $\mu$ be a $p$-adic missing digits measure in $\mathbb{R}$ with $t$ missing digits. Let $s=\Haus \mu$. Then for all large enough $p$ so that $\dim_{l^1}\mu>1/2$, there is a $c=c_{t,p}>0$ such that for the approximation function $\psi(q)=q^{-1-c'}, c'\in (0,c],$ we have
	\[
	s-\frac{c'}{2+c'}\geq \Haus W(\psi,0)\cap supp(\mu)\geq  s-o_{c'\to 0}(1).
	\]
	Thus in particular, as $c'\to 0,$ we have
	\[
	\Haus W(\psi,0)\cap supp(\mu)\to s.
	\]
\end{thm}
\begin{rem}
	We crucially need the homogeneous condition $\gamma=0$ in the proof. However, we suspect the inhomogeneous version ($\gamma\neq 0$) of this theorem holds as well. This theorem applies to the middle-$15$th Cantor measure.
\end{rem}
The lower bound for the Hausdorff dimension in Theorem \ref{Jarnik2} is rather unsatisfactory. In fact if $p$ is really large enough, we have the following improvement.
\begin{thm}\label{Jarnik4}
	Let $t\geq 1$ be an integer. Let $p\geq 3$ be an integer. Let $\mu$ be a $p$-adic missing digits measure in $\mathbb{R}$ with $t$ missing digits. Let $s=\Haus \mu$. Then for all large enough $p$, $\mu$ is spectral and
	\[
	\dim_{l^1}\mu \Haus \mu>1/2.
	\] Furthermore, there is a $c=c_{t,p}>0$ such that for the approximation function $\psi(q)=q^{-1-c'}, c'\in (0,c],$ we have
	\[
	\Haus W(\psi,0)\cap supp(\mu)=s-\frac{c'}{2+c'}.
	\]
\end{thm}
\begin{rem}
	The difference is in the condition 'for all large enough $p$'. In fact, for this theorem, the requirement on $p$ is stronger than in Theorem \ref{Jarnik2}. Our numerics in Sections \ref{Simple},\ref{Simple2} are not strong enough to test whether this theorem applies for the middle 15th Cantor measure. For $t=1,$ it can be checked that this theorem applies for $p>10^7.$
\end{rem}
\begin{proof}[Proof of Theorem \ref{Jarnik2}]
	The upper bound is a special case of Theorem \ref{Jarnik1}. We now focus on the lower bound. 
	\subsubsection*{Step 1: Local counting property}
	We will now examine in details how scaling will affect the counting property. Let $\mu'$ be a branch of $\mu$. Then $\mu'$ is a probability measure which is a translated copy of $\mu$ scaled down by $p^{-k}$ where $k\geq 0$ is an integer. Thus the norm of the Fourier coefficients of $\mu'$ is related to that of $\mu$ by
	\[
	|\hat{\mu'}(\xi)|=|\hat{\mu}(p^{-k} \xi)|.
	\]
	Here, we recall the formula for Fourier transform,
	\[
	\hat{\mu}(\xi)=\int_{\mathbb{R}} e^{-2\pi i \xi x}d\mu(x).
	\]
	Thus we see from Theorem \ref{Lattice Counting} and the argument in Section \ref{LCST},
	\begin{align*}
		\tag{B}\sum_{q=Q}^{2Q}\mu'(A(\delta,Q,0))\leq c_1 Q\delta(1+Q^{-1}O(\sum_{\xi\in\mathbb{Z}, |\xi|\leq 2Q/\delta} d(\xi)|\hat{\mu}(\xi/p^k)|)),
	\end{align*}
	where the implied constant in the $O()$ terms does not depend on the choice of $\mu'.$ We need to sample the Fourier transform with spacing $p^{-k}$ instead of $1.$ Sampling with spacing $p^{-k}$ will not create too much more difficulties than just sampling with spacing $1,$ see Section \ref{Bonus}. More precisely, we had a decay estimate for $|\hat{\mu}(\xi)|$ which depends on the integer part of $\xi.$ Thus, if we sample $\hat{\mu}$ with spacing $p^{-k},$ then inside each integer interval (i.e. $[l,l+1],l\in\mathbb{Z}$) we have the same estimate for $\hat{\mu}$ repeated $O(p^k)$ times. In particular, we see that for each $1/2<s_1< \dim_{l^1}\mu,$ 
	\[
	\sum_{\xi\in\mathbb{Z}, |\xi|\leq 2Q/\delta} |\hat{\mu}(\xi/p^k)|\ll p^k ((Q/p^k\delta)+1)^{1-s_1}.
	\] 
	More precisely, there is a number $M>1$ (which does not depend on $k$) such that for $Q>M \times p^{k},$
	\[
	\sum_{\xi\in\mathbb{Z}, |\xi|\leq 2Q/\delta} d(\xi)|\hat{\mu}(\xi/p^k)|\leq M \times p^{s_1 k} Q^{1-s_1}\delta^{s_1-1}.
	\]
	For the problem we are considering, we need to choose $\delta=(1/Q)^{1+c'}$ for $c'>0.$ Let $\beta$ be a number such that $\beta<(2s_1-1-c'(1-s_1))/s_1.$ We also want to require that
	\[
	\beta>0.
	\]
	As $s_1>1/2,$ this can be achieved as long as $c'$ is small enough. Namely, there is a number $c>0$ such that as long as $c'<c,$ it is possible to achieve that $\beta>0.$ We choose one such value for $\beta\in (0,1).$ The choice can depend on $c'.$  In case $s_1$ is only a little bit larger than $1/2$ then $\beta$ has to be very small. We see that as long as $Q$ is much larger than $p^{k/\beta}$ we will have
	\[
	\sum_{\xi\in\mathbb{Z}, |\xi|\leq 2Q/\delta} d(\xi)|\hat{\mu}(\xi/p^k)|\leq Mp^{s_1k}Q^{1-s_1-(s_1-1)(1+c')}=MQ p^{s_1 k}Q^{-(2s_1-1-c'(1-s_1))}
	\]
	being much smaller than $Q.$ More precisely,  if $Q\geq 100^{s_1/\beta} M^{s_1/\beta} p^{k /\beta}$ we have
	\[
	Q^{-1}\sum_{\xi\in\mathbb{Z}, |\xi|\leq 2Q/\delta}d(\xi) |\hat{\mu}(\xi/p^k)|\leq \frac{1}{100}.
	\]
	Thus from the Estimate (B) in above, we see that there is a number $M'$ such that for $Q\geq M' p^{k /\beta}$
	\[
	\sum_{q=Q}^{2Q}\mu'(A(\delta,Q,0))\leq 2c_1 Q\delta.
	\]
	Here $k$ is the scaling factor between $\mu'$ and $\mu.$ The number $M'$ does not depend on $k.$ Similarly, we can also obtain a lower bound as in the proof of Theorem \ref{GCPtoKS}. Thus for $Q\geq M''p^{k/\beta}$ with constants $M''$
	\[
	c_1 Q\delta/2\leq \sum_{q=Q}^{2Q}\mu'(A(\delta,Q,0))\leq 2c_1 Q\delta.
	\]
	We can increase the values for $M', M''$ if necessary. We simply assume $M'=M''.$ Next, for technical reasons, we need to have an lower estimate for 
	\[
	\sum_{q=Q, q\text{ is prime}}^{2Q}\mu'(A(\delta,Q,0)).
	\]
	Recall Section \ref{Prime}, it is possible to see that by the prime number theorem, for $Q\geq M'p^{k/\beta}$ (again, increase the values for $M'$ if necessary, the key point here is that we have the same $\beta$ as above),
	\[
	\frac{c_1}{2}\delta\frac{Q}{\log Q}\leq \sum_{q=Q, q\text{ is prime}}^{2Q}\mu'(A(\delta,Q,0))\leq 2c_1\delta\frac{Q}{\log Q}.
	\]
	Here, the condition $Q\geq M'p^{k/\beta}$ is the main obstruction in our method to show that $o_{c'\to 0}(1)$ can be chosen to be the conjectured $O_{c'\to 0}(c').$ In fact, we are only able to show that the error term is $O_{c'\to 0}(\sqrt{c'}).$ From Example \ref{15th}, we see that the above arguments hold for the middle-15th Cantor measure. 
	\subsubsection*{Step 2: A Cantor set construction}
	We now want to proceed with the estimate of the Hausdorff dimension. The idea is to construct a compact subset of the limsup set under consideration. This will be done via a Cantor set construction. This strategy also appeared in the proof of the mass transference principle in \cite{BV}.

	Using GCP(1) for $\mu$, we see that there is a $c>0$ (which matches the $c$ appeared before without loss of generality) such that for $\delta_Q=Q^{-1-c'}$, $c'\in (0,c]$
	\[
	\sum_{q=Q}^{2Q}\mu(A(\delta_Q,q,0))\asymp Q\delta_Q
	\]
	and
	\[
	\sum_{q=Q,q\text{ is prime}}^{2Q}\mu(A(\delta_Q,q,0))\asymp Q\delta_Q/\log Q.
	\]
	Using the AD-regularity of $\mu,$ we have for $q\in [Q,2Q],$
	\begin{align*}
		&	\#\{\mathbf{k}\in\mathbb{Z}: d(\mathbf{k}/Q, supp(\mu)) \leq \delta/2   \}&\\&\ll \mu(A(\delta_Q,q,0))/(\delta_Q/Q)^s&\\&\ll \#\{\mathbf{k}\in\mathbb{Z}: d(\mathbf{k}/Q, supp(\mu)) \leq 2\delta   \}.&
	\end{align*}
	Let $B_Q=\bigcup_{q=Q, q\text{ is prime}}^{2Q} A(\delta_Q,q,0).$ Two rational numbers with denominators in $[Q,2Q]$ are either equal or have difference at least $1/(4Q^2).$\footnote{This is where we need $\gamma=0.$} In particular, if the denominators are primes numbers, then unless the two rational numbers are $0$ (or 1), otherwise, they separate with distance at least $1/(4Q^2).$ From here for large enough $Q$, it is possible to see that $B_Q$ is a disjoint union of balls centred at rational numbers with prime denominators in $[Q,2Q].$ Now, the contribution of the balls at $0$ (if at all exist) to the sum
	\[
	\sum_{q=Q, q\text{ is prime}}^{2Q}\mu(A(\delta_{Q},q,0))
	\]
	is
	\[
	\ll (\delta_Q/Q)^{s}\times \frac{Q}{\log Q}.
	\]
	Thus as long as 
	\[
	\delta_Q\gg Q^{-s/(1-s)},
	\]
	we see that
	\[
	(\delta_Q/Q)^{s}\times \frac{Q}{\log Q}
	\]
	is much smaller than
	\[
	Q\delta_Q/\log Q.
	\]
	Therefore the overlapping balls at $0$ create no problem to us.\footnote{Here $s$ is necessarily larger than $1/2$ because we need to consider spectral and thick missing digits measures.} 
	
	Our goal now is to find a lower bound for the Hausdorff dimension of $\limsup_{Q\to\infty} B_Q$ which is a lower bound for the Hausdorff dimension of $W(\psi,0)\cap supp(\mu).$ First, let $Q_1$ be a sufficiently large integer such that 
	\[
	\sum_{q=Q_1, q\text{ is prime}}^{2Q_1}\mu(A(\delta_{Q_1},q,0))\geq c_1 Q_1\delta_{Q_1}/\log Q_1,
	\]
	where $\delta_Q=Q^{-1-c'}$, $c'$ is a fixed number in $(0,c]$ and $\alpha_1>0$ is a constant. Consider $B_Q,$ which is a union of disjoint balls with radius in $[\delta_{Q_1}/Q_1,2\delta_{Q_1}/Q_1]$ (increase the value of $Q_1$ if necessary). Let $B$ be one of those balls which intersect $supp(\mu).$ We denote the collection of such balls as $\mathcal{B}_{Q_1}.$ By what we have seen, there are
	\[
	\geq \alpha'_1 Q_1\delta_{Q_1}\left(\frac{Q_1}{\delta_{Q_1}}\right)^s/\log Q_1
	\]
	such balls where $\alpha'_1$ is another constant. We consider $2B,$ the ball with the same centre as $B$ but double the radius.  We denote 
	\[\mathcal{K}_1=\bigcup_{B\in\mathcal{B}_{Q_1}} 2B.\]
	Again, by choosing $Q_1$ to be large enough, the above union is a disjoint union. Since $B\cap supp(\mu)\neq\emptyset,$ we see that $2B$ contains the support of a branch of $\mu.$ More precisely, there is a branch $\mu'$ of $\mu$ whose support is a scaled copy of $supp(\mu).$ Moreover, the length of the convex hull of $supp(\mu')$ is at least
	\[
	\frac{\delta_{Q_1}}{2pQ_1}.
	\]
	In fact, the convex hull of $supp(\mu)$ is $[0,1]$ and the convex hull of $supp(\mu')$ is some $p$-adic interval, i.e. an interval of length $p^{-k},k\in\mathbb{N}$ whose end points are rationals with denominator $p^{-k}.$ As $\mu'$ is again spectral and thick, it has GCP(1). Moreover, by the first step of this proof, we can choose a large enough $Q_2$ (which can be chosen to be  $\asymp(Q_1/\delta_{Q_1})^{1/\beta}$) such that
	\[
	\sum_{q=Q_2, q\text{ is prime}}^{2Q_2} \mu'(A(\delta_{Q_2}, q,0))\asymp Q_2 \delta_{Q_2}/\log Q_2.
	\]
	In particular, we have
	\[
	\sum_{q=Q_2,q\text{ is prime}}^{2Q_2} \mu'(A(\delta_{Q_2}, q,0))\geq c_1 Q_2\delta_{Q_2}/\log Q_2.
	\]
	$\mu'$ is a scaled copy of $\mu$ and it is $s$-dimensional AD-regular as well. The scaling factor is $p^k$ and for all $x\in supp(\mu')$ and all small enough $R,$ $\mu'(B(x,R))\asymp p^{ks} R^s.$ Therefore we have for a constant $\alpha'_2>0$ (depends on $Q_1$, which can be chosen to be $\alpha'_1 (\delta_{Q_1}/Q_1)^s$) such that $2B$ contains
	\[
	\geq \alpha'_2Q_2\delta_{Q_2} \left(\frac{Q_2}{\delta_{Q_2}}\right)^s/\log Q_2
	\]
	many disjoint balls of radius in $[\delta_{Q_2}/Q_2,2\delta_{Q_2}/Q_2]$ which also intersect $supp(\mu').$ We can perform the above steps for each $B\in \mathcal{B}_{Q_1}$ and obtain smaller disjoint balls. We collect all the smaller disjoint balls we obtained, double them and then denote their union as $\mathcal{K}_2.$ We can go on constructing $\mathcal{K}_3,\mathcal{K}_4,\dots.$  We see that
	\[
	\mathcal{K}_1\supset \mathcal{K}_2\supset \mathcal{K}_3\dots
	\]
	There are integers $Q_1<Q_2<Q_3\dots$ and positive numbers $\alpha'_1, \alpha'_2,\alpha'_3,\dots$ such that for $j\geq 1,$
	$
	\mathcal{K}_j
	$
	is a disjoint union of at least
	\[
	\prod_{i=1}^{j} \alpha'_j Q_j\delta_{Q_j} \left(\frac{Q_j}{\delta_{Q_j}}\right)^s/\log Q_j
	\]
	many balls of radius in $[2\delta_{Q_j}/Q_j,4\delta_{Q_j}/Q_j].$ Also, $\alpha'_1$ is an absolute constant, and $\alpha'_j$ depends on $Q_{j-1}$ for $j\geq 2.$ Now we will choose for $j\geq 2,$ $Q_{j}=[ (Q_{j-1}/\delta_{Q_{j-1}})^{1/\beta''}].$ Where $0<\beta''<\beta$ can be chosen to be arbitrarily small. Let $\mathcal{K}(Q_1,Q_2,\dots)=\cap_{j\geq 1}\mathcal{K}_j.$ This is a compact set. 
	\subsubsection*{Step 3: A Borel probability measure $\nu$}
	We now analyse the set $\mathcal{K}.$ Ultimately, we want to estimate its Hausdorff dimension. The set $\mathcal{K}$ was obtained by a certain Cantor set construction. The general formula for the Hausdorff dimension should be known. However, we have not found a sufficiently convenient formula to use. Therefore we provide details for a lower estimate of $\Haus \mathcal{K}.$ We do this via the mass distribution principle, lemma \ref{MDP}. We will construct a probability measure supported by $\mathcal{K}$ in this step and study its Frostman property in the next step.
	
	 First,  for each $j\geq 1,$ $\mathcal{K}_{j}$ is disjoint union of balls. We denote $\mathcal{D}_1$ as the collection of the disjoint balls in $\mathcal{K}_1.$ We enumerate them by $B_1,\dots,B_{n_1}$ where $n_1=\#\mathcal{D}_1.$ For each $B_j, j\in\{1,\dots, n_1\},$ we see that $B_j\cap \mathcal{K}_2$ is a disjoint union of $n_{2,j}$ many balls. We enumerate them by $B_{j,1},\dots,B_{j,n_{2,j}}.$ By dropping some of them we can assume that there is a number $n_2=n_{2,j}$ for each $j$ and $n_2\geq \alpha'_2Q_2\delta_{Q_2}(Q_2/\delta_{Q_2})^s.$  This procedure can go on indefinitely and we obtained a coding system for a subset of $\mathcal{K}.$ This is in some sense a homogeneous version of $\mathcal{K}.$ We still use $\mathcal{K}$ to denote this subset.  Now we assign a Borel probability measure on $\mathcal{K}$ by giving measures for 'cylinder sets', i.e. sets of form $B_{j_1\dots j_k}, k\in\mathbb{N}$ and for each $i\in\{1,\dots,k\},$ $j_i\in \{1,\dots, n_{i}\}.$ We define
	\[
	\nu(B_{j_1\dots j_k})=\frac{1}{n_1\dots n_k}.
	\]
	\subsubsection*{Step 4: A Frostman property of $\nu$}
	Now let $r\in (0,1)$ be a number. We want to estimate $\nu(B_r)$ for any $r$-ball $B_r.$ We assume that $r$ is small such that there is a number $j\geq 4$ such that 
	\[
	\frac{\delta_{Q_{j}}}{Q_{j}}<r\leq \frac{\delta_{Q_{j-1}}}{Q_{j-1}}.
	\]
	Then $B_r$ intersects at most one ball forming $\mathcal{K}_{j-1}.$ Since balls forming $\mathcal{K}_j$ are separated with distance at least $1/4Q^2_j,$ we first see that $B_r$ intersects at most
	\[
	1+16 r Q^2_j
	\]
	many of them.  However, balls forming $\mathcal{K}_j$ are $\delta_{Q_j}/Q_j$ close to $supp(\mu).$ In particular, we see that $B_r$ intersects at most
	\[
	100(1+16 r Q^2_j)^s
	\]
	many balls forming $\mathcal{K}_j.$ From here we see that
	\[
	\log \nu(B_r)\leq \log \left(100(1+16rQ^2_j)^s \frac{1}{n_1\dots n_j}\right)=\log 100(Q^{-2}_j+16r)^s+\log \frac{Q^{2s}_j}{n_1\dots n_j}.
	\]
	We want to estimate
	\[
	\frac{\log \nu(B_r)}{\log r}.
	\]
	Recall that $\alpha'_{j+1}=\alpha'_{j}(\delta_{Q_j}/Q_{j})^s$ and $Q_{j+1}=[(Q_{j}/\delta_{Q_j})^{1/\beta''}]$ for $j\geq 1.$ Then for large enough $j,$ we define a number $A$ to be equal to
	$
	\log n_1\dots n_j/Q^{2s}_j.
	$
	The we have
	\begin{align*}
		&A\geq \log Q^{-2s}_j\prod_{i=1}^{j} \alpha'_j Q_j\delta_{Q_j} \left(\frac{Q_j}{\delta_{Q_j}}\right)^s\frac{1}{\log Q_j}\\&= O_{r\to 0}(1)+\log \prod_{i=1}^j \frac{1}{\log Q_j}+\\&\log (\alpha'_1)^j \frac{(Q^{2+c'}_1\dots Q^{2+c'}_j)^s}{Q^{c'}_1\dots Q^{c'}_j}\frac{Q_j^{-2s}}{((Q^{2+c'}_1)^{j-1}(Q^{2+c'}_2)^{j-2}\dots (Q^{2+c'}_{j-1}))^s}&
		\\&=O_{r\to 0}(1)+\log \prod_{i=1}^j \frac{1}{\log Q_j}+\\&\log (\alpha'_1)^j \frac{(Q^{2+c'}_1\dots Q^{2+c'}_{j-1})^s}{Q^{c'}_1\dots Q^{c'}_{j-1}}\frac{Q_j^{c's-c'}}{((Q^{2+c'}_1)^{j-1}(Q^{2+c'}_2)^{j-2}\dots (Q^{2+c'}_{j-1}))^s}&
		\\&=O_{r\to 0}(1)+\log \prod_{i=1}^j \frac{1}{\log Q_j}+\log (\alpha'_1)^j+\\&\log \frac{1}{Q^{c'}_1\dots Q^{c'}_{j-1}}\frac{Q_j^{c's-c'}}{((Q^{2+c'}_1)^{j-2}(Q^{2+c'}_2)^{j-3}\dots (Q^{2+c'}_{j-2}))^s}.&
	\end{align*}
	The very last term is smaller than zero. Thus we see that
	\begin{align*}\tag{BB}
		\frac{-A}{\log r}&\geq \frac{-O_{r\to 0}(1)-j \log \alpha'_1-j\log\log Q_j}{\log r}+\frac{\log \frac{1}{Q^{c'}_1\dots Q^{c'}_{j-1}}\frac{Q_j^{c's-c'}}{((Q^{2+c'}_1)^{j-2}(Q^{2+c'}_2)^{j-3}\dots (Q^{2+c'}_{j-2}))^s}}{\log (Q_{j-1}/\delta_{Q_{j-1}})}\\&\geq \frac{-O_{r\to 0}(1)-j \log \alpha'_1-j\log\log Q_j}{\log r}+\\&\frac{\log \frac{1}{Q^{c'}_1\dots Q^{c'}_{j-1}}\frac{1}{((Q^{2+c'}_1)^{j-2}(Q^{2+c'}_2)^{j-3}\dots (Q^{2+c'}_{j-2}))^s}}{\log (Q_{j-1}/\delta_{Q_{j-1}})}+\frac{c's-c'}{\beta''}.&
	\end{align*}
	As $1/r\geq Q_{j-1}/\delta_{Q_{j-1}},$ we see that
	\[
	|\log r|\geq (2+c')\log Q_{j-1}.
	\] 
	Observe that for integers $l,$
	\[\log Q_{l}=(1+o_{l\to\infty}(1))\left(\frac{2+c'}{\beta''}\right)^{l-1} \log Q_1.
	\]
	This implies that for $j\to\infty,$
	\[
	\frac{j}{\log r}\to 0
	\]	
	and
	\[
	\frac{j\log\log Q_j}{\log r}\to 0.
	\]
	Thus the first term in the last line of (BB) tends to $0$ for $r\to 0.$ The second term is complicated. We see that
	\begin{align*}
		&\frac{\log \frac{1}{Q^{c'}_1\dots Q^{c'}_{j-1}}\frac{1}{((Q^{2+c'}_1)^{j-2}(Q^{2+c'}_2)^{j-3}\dots (Q^{2+c'}_{j-2}))^s}}{\log (Q_{j-1}/\delta_{Q_{j-1}})}\\&=\frac{\log Q^{-c'}_{j-1}}{\log (Q^{2+c'}_{j-1})}+\frac{\log \frac{1}{Q^{c'}_1\dots Q^{c'}_{j-2}}\frac{1}{((Q^{2+c'}_1)^{j-2}(Q^{2+c'}_2)^{j-3}\dots (Q^{2+c'}_{j-2}))^s}}{\log (Q_{j-1}/\delta_{Q_{j-1}})}\\
		&=\frac{-c'}{2+c'}+\frac{\log \frac{1}{Q^{c'}_1\dots Q^{c'}_{j-2}}\frac{1}{((Q^{2+c'}_1)^{j-2}(Q^{2+c'}_2)^{j-3}\dots (Q^{2+c'}_{j-2}))^s}}{\log (Q_{j-1}/\delta_{Q_{j-1}})}.
	\end{align*}
	As $Q_{j+1}=[(Q_j/\delta_{Q_j})^{1/\beta''}],$ we see that the absolute value of the second term in the last line above is at most
	\[
	\frac{c'(a+a^2+\dots)+s(2+c')(a+2a^2+3a^3+\dots)}{2+c'}
	\]
	where $a=\beta''/(2+c')<1/2.$ Let
	\[
	H=1+2\frac{1}{2}+3\frac{1}{2^2}+\dots<\infty.
	\]
	Then we see that ($s<1$)
	\[
	c'(a+a^2+\dots)+s(2+c')(a+2a^2+3a^3+\dots)\leq Hc'\frac{\beta''}{2+c'}+sH\beta''\leq 2H\beta''.
	\]
	Collecting all the estimates for the terms in (BB), we see that
	\[
	\frac{-A}{\log r}\geq \frac{-c'}{2+c'}+o_{r\to 0}(1)-\frac{c'-c's}{\beta''}-\frac{2H\beta''}{2+c'}.\tag{BBB}
	\]

	Now we need to consider 
	\[
	\frac{\log 100(Q^{-2}_j+16r)^s}{\log r}=s+\frac{\log 100(16+(Q^{-2}_jr^{-1}))^s}{\log r}.
	\]
	If $r>Q^{-2}_j,$ then 
	\[
	\left|\frac{\log 100(16+(Q^{-2}_jr^{-1}))^s}{\log r}\right|\leq \frac{\log (100\times 20)}{|\log r|}.
	\]
	Thus we see that
	\[
	\frac{\log \nu(B_r)}{\log r}\geq s-\frac{c'}{2+c'}-\frac{c'-c's}{\beta''}-\frac{2H\beta''}{2+c'}+o_{r\to 0}(1).
	\]
	If $r\leq Q^{-2}_{j},$ then $B_r$ intersects at most $100$ many of balls in $\mathcal{K}_j.$ From here we see that
	\[
	\nu(B_r)\leq 100 \frac{1}{n_1\dots n_j}.
	\]
	Then we have
	\[
	\frac{\log \nu(B_r)}{\log r}\geq \frac{\log 100+\log n_1\dots n_j}{|\log r|}\geq \frac{\log 100 +\log n_1\dots n_j}{\log (Q_j^{2+c'})}.
	\]
	As we did in above
	\[
	\log n_1\dots n_j\geq O_{r\to 0}(1)+\log (\alpha'_1)^j \frac{(Q^{2+c'}_1\dots Q^{2+c'}_j)^s}{Q^{c'}_1\dots Q^{c'}_j}\frac{1}{((Q^{2+c'}_1)^{j-1}(Q^{2+c'}_2)^{j-2}\dots (Q^{2+c'}_{j-1}))^s}.
	\]
	From here, we can argue as in above to see that
	\begin{align*}
		\frac{\log n_1\dots n_j}{\log (Q^{2+c'}_j)}&\geq o_{r\to 0}(1)+s-\frac{c'}{2+c'}-\frac{2H\beta''}{2+c'}.
	\end{align*}
	Therefore we see that
	\[
	\frac{\log \nu(B_r)}{\log r}\geq s-\frac{c'}{2+c'}-\frac{c'-c's}{\beta''}-\frac{2H\beta''}{2+c'}+o_{r\to 0}(1).
	\]
	Thus by the mass distribution principle (Lemma \ref{MDP}), we see that
	\[
	\Haus \mathcal{K}\geq s-\frac{c'}{2+c'}-\frac{c'-c's}{\beta''}-\frac{2H\beta''}{2+c'}.
	\]
	We have now the freedom to choose the value of $\beta''.$ Since the requirement of $\beta''$ is that $0<\beta''<\beta,$ it is possible to see that
	\[
	\Haus \mathcal{K}\geq s-O_{c'\to 0}(\sqrt{c'}).
	\]
	This concludes the proof.
\end{proof}

\begin{proof}[Proof of Theorem \ref{Jarnik4}]
	Let $\rho<s-c'/(2+c')$ be a positive number. Let $r>0$ be a small number. We consider a countable collection of  intervals $\mathcal{C}=\mathcal{C}(r,\rho)$ with 
	\[
	\sum_{I\in\mathcal{C}}|I|^\rho<1.
	\]
	Our goal is to show that $\cup_{I\in\mathcal{C}}I$ does not contain $supp(\mu)\cap W(\psi,0)$ as long as $r$ is small enough. Assuming this for now, we see that by definition, $\Haus W(\psi,0)\cap supp(\mu)\geq \rho.$ As $\rho$ can be chosen to be close to $s-c'/(2+c'),$ the result concludes.
	
	We follow an argument by Besicovitch (\cite{B34}). We will choose an increasing sequence of integers $10<Q_1<Q_2<\dots$ satisfying some properties which will be clarified later. First, we fix a number $M=M(c',\mu,\rho).$ Then according to this value, we can determine $Q_1=Q_1(M).$  After this, we can determine $Q_{i+1}=Q_{i+1}(Q_i)$ inductively.
	
	Let $r\leq (1/Q_1)^{2+c'}.$ Let $\mathcal{F}_2$ be the collection of intervals of length $2(1/Q_2)^{2+c'}$ centred at each rational numbers with prime denominators in $[Q_2,2Q_2]$ and which intersect $supp(\mu).$  Using GCP(1), we see that there is a $c>0$ such that for $\delta_Q=Q^{-1-c'}$, $c'\in (0,c]$
	\[
	\sum_{q=Q, q\text{ is prime}}^{2Q}\mu(A(\delta_Q,q,0))\asymp Q\delta_Q/\log Q.
	\]Using the AD-regularity of $\mu,$ we have for $q\in [Q,2Q],$
	\begin{align*}
		&\#\{\mathbf{k}\in\mathbb{Z}: d(\mathbf{k}/Q, supp(\mu)) \leq \delta/2 \} & \\
		&\ll\mu(A(Q^{-2-c'},q,0))/(Q^{-1-c'}/Q)^s&\\&\ll \#\{\mathbf{k}\in\mathbb{Z}: d(\mathbf{k}/Q, supp(\mu)) \leq 2\delta   \}.&
	\end{align*}
	Thus we see that
	\[
	\#\mathcal{F}_2\gg \frac{Q_2\delta_{Q_2}}{\log Q_2}\left(\frac{1}{Q^{2+c'}_2}\right)^s\gg\frac{Q_2^{2s+c'(s-1)}}{\log Q_2}.\tag{I}
	\]
	
	First, let us consider the intervals in $\mathcal{C}$ whose lengths are at least $1/Q^{\beta}_2.$ (Here $\beta$ is as in the (Step 1) in the proof of Theorem \ref{Jarnik2}. It can be chosen according to $s_1,c'$. We will later discuss this matter.)  Denote this collection as $\mathcal{C}_1.$ Then we see that
	\[
	\sum_{I\in\mathcal{C}_1}|I|^{s}=\sum_{I\in\mathcal{C}_1}|I|^{\rho}|I|^{s-\rho}<\frac{1}{Q^{(2+c')(s-\rho)}_1}.
	\]
	For each interval $I\in\mathcal{C}_1,$ we see that there are $\ll |I|^s Q^{2s+c'(s-1)}_2/\log Q_2$ many intervals in $\mathcal{F}_2$ intersecting $I.$ This is because $Q_2\geq |I|^{1/\beta}$ and the local counting property in (Step 1) applies. Thus the number of intervals in $\mathcal{F}_2$ intersecting intervals in $\mathcal{C}_1$ is
	\[
	\ll \sum_{I\in\mathcal{C}_1}  |I|^sQ^{2s+c'(s-1)}_2/\log Q_2\leq \frac{Q^{2s+c'(s-1)}_2}{\log Q_2}\frac{1}{Q^{(2+c')(s-\rho)}_1}.\tag{*}
	\]
	
	Next, we consider the intervals in $\mathcal{C}$ with length in between $1/(Q^\beta_2)$ and $1/Q^2_2.$\footnote{We note that $\beta<1.$} We denote this collection as $\mathcal{C}_2.$ We can not use the local counting property as above. However, since the intervals in $\mathcal{F}_2$ have centres which are at least $1/4Q^2_2$ away from each other, we see that for each $I\in\mathcal{C}_2$, there are
	\[
	\ll (|I|Q^2_2)^s
	\]
	many intervals in $\mathcal{F}_2$ that can intersect $I.$ Thus the number of intervals in $\mathcal{F}_2$ intersecting intervals in $\mathcal{C}_2$ is
	\[
	\ll  \sum_{I\in\mathcal{C}_2} (|I|Q^2_2)^s\leq Q^{2s}_2\sum_{I\in\mathcal{C}_2}|I|^{s-\rho}|I|^{\rho}\leq Q^{2s}_2\frac{1}{Q^{\beta(s-\rho)}_2}.\tag{**}
	\]
	
	Finally, we consider the collection $\mathcal{C}_3$ of intervals in $\mathcal{C}$ with lengths at least  $1/Q^{2+c'}_2$ and at most $1/Q^2_2.$  We see that each interval $I\in\mathcal{C}_3$ can intersect $O(1)$ many intervals in $\mathcal{F}_2.$ Here the $O(1)$ term is an absolute constant. Denote $F_2$ as the number of intervals in $\mathcal{F}_2$ which also intersect intervals in $\mathcal{C}_3.$ We see that
	\[
	F_2 \left(\frac{1}{Q^{2+c'}_2}\right)^\rho\leq O(1) \sum_{I\in\mathcal{C}_3} |I|^\rho=O(1).
	\]
	The above is deduced by a double counting argument. The leftmost term comes from the following consideration. For each interval in $\mathcal{F}_2$ which also intersects intervals in $\mathcal{C}_3,$ we choose one such intersecting interval in $\mathcal{C}_3$ and count its length powered by $\rho.$ We then do this for each of those $F_2$ many intervals in $\mathcal{F}_2$ and in total we have counted a quantity which is at least the leftmost side of the inequality above. Each interval in $\mathcal{C}_3$ is involved in at most $O(1)$ times in the above counting argument because it can intersect $O(1)$ many intervals in $\mathcal{F}_2.$ This gives the above inequality.
	
	Thus we see that
	\[
	F_2\ll Q^{(2+c')\rho}_2.\tag{***}
	\]
	
	 From (*),(**) and (***), we conclude that  the number of intervals in $\mathcal{F}_2$ which intersect intervals in $\mathcal{C}$ with length at least $1/(Q^{2+c'}_2)$ is 
	 \begin{align*}
	 &\ll \frac{Q^{2s+c'(s-1)}_2}{\log Q_2}\frac{1}{Q^{(2+c')(s-\rho)}_1}+Q^{2s}_2\frac{1}{Q^{\beta(s-\rho)}_2}+Q^{(2+c')\rho}_2&\\
	 &= \frac{Q^{2s+c'(s-1)}_2}{\log Q_2}\left( \frac{1}{Q^{(2+c')(s-\rho)}_1}+\frac{Q^{c'(1-s)}_2\log Q_2}{Q^{\beta(s-\rho)}_2}+\frac{Q^{(2+c')\rho}_2\log Q_2}{Q^{2s+c'(s-1)}_2}\right).&\tag{@}
	 \end{align*}
	 Now, we want to pose the condition that (so that the second term in the bracket has negative power on $Q_2$)
	 \[
	 \beta(s-\rho)>c'(1-s)\tag{Condition 1}.
	 \]
	 We will discuss about this condition at the end of the proof. In particular, we will show that this condition is not empty, i.e. it can be satisfied.  Observe that $\rho<s-c'/(2+c').$ We see that
	 \[
	 2s+c'(s-1)>(2+c')\rho.
	 \]
	 Thus, there is a (possibly very large) number $M>2/\beta$ such that if $Q_2>Q^M_1,$ we have
	 \[
	 \frac{Q^{c'(1-s)}_2\log Q_2}{Q^{\beta(s-\rho)}_2}+\frac{Q^{(2+c')\rho}_2\log Q_2}{Q^{2s+c'(s-1)}_2}<\frac{1}{Q^{(2+c')(s-\rho)}_1}.
	 \] 
	 Then we can write (@) in a bit easier way,
	 \[
	 \frac{Q^{2s+c'(s-1)}_2}{\log Q_2} \frac{2}{Q^{(2+c')(s-\rho)}_1}.\tag{@'}
	 \]
	 Let $\mathcal{F}'_2$ be the intervals that do not survive after the above step. Namely, we consider those intervals in $\mathcal{F}_2$ which intersect intervals in $\mathcal{C}$ with length at least $1/Q^{2+c'}_2.$ The above arguments tell us that there are numbers $M,M'>0$ such that for $Q_2>Q^M_1,$
	 \[
	 \#\mathcal{F}'_2\leq \frac{Q^{2s+c'(s-1)}_2}{\log Q_2}\frac{M'}{Q^{(2+c')(s-\rho)}_1}.\tag{II}
	 \]
	 This is much smaller than $\#\mathcal{F}_2$. In fact, we can choose $Q_1$ to be large enough so that $\#\mathcal{F}'_2$ is at most $0.01\#\mathcal{F}_2.$ We fix such a value for $Q_1.$ After this step,  let $\mathcal{F}''_2=\mathcal{F}_2\setminus \mathcal{F}'_2.$ Of course, $\mathcal{F}''_2$ depends on $Q_2.$ We will choose a value for $Q_2$ in the next paragraph.
	 
	 Our aim is to iterate the above argument. 
	Let $Q_3>Q^M_2$ be an integer.  Let $\mathcal{F}_3$ be the collection of intervals of length $2(1/Q_3)^{2+c'}$ centred at each rational numbers with prime denominators in $[Q_3,2Q_3]$ and which intersect $supp(\mu).$  Now, we take $\mathcal{F}_{3,1}\subset\mathcal{F}_3$ be the collection of intervals in $\mathcal{F}_3$ which are contained intervals in $\mathcal{F}''_2.$ By the local counting property in (Step 1), we see that as long as $Q_3\gg Q^{1/\beta}_2,$ we have
	\[
	\#\mathcal{F}_{3,1}\gg \#\mathcal{F}''_2 \left(\frac{1}{Q^{2+c'}_2}\right)^s\frac{Q_3^{2s+c'(s-1)}}{\log Q_3}\gg \frac{1}{Q^{c'}_2\log Q_2}\frac{Q_3^{2s+c'(s-1)}}{\log Q_3}.
	\]
	For the last inequality, we have used the fact that $\#\mathcal{F}'_2$ is much smaller than $\#\mathcal{F}_2$ (see (I), (II)). Because we forced $M>2/\beta,$ we see that $Q_3\gg Q^{1/\beta}_2$ is automatically satisfied if $Q_3\geq Q^M_2.$ Now we apply the arguments before with $Q_2,Q_3$ replacing $Q_1,Q_2.$ We have considered intervals in $\mathcal{C}$ with length at least $1/(Q^{2+c'}_2).$ So we are indeed in the place of applying those arguments. As a result, we can find a collection $\mathcal{F}'_3$ of intervals in $\mathcal{F}_3$ which intersect intervals in $\mathcal{C}$ with length in between $1/Q^{2+c'}_2$ and $1/Q^{2+c'}_3.$ We see that
	\[
	\#\mathcal{F}'_3\leq \frac{Q^{2s+c'(s-1)}_3}{\log Q_3}\frac{M'}{Q^{(2+c')(s-\rho)}_2}.
	\]
	Since $s-\rho>c'/(2+c')$, we see that
	$
	Q^{(2+c')(s-\rho)}_2
	$ is much larger than $Q^{c'}_2.$ Thus $\#\mathcal{F}'_3$ is much smaller than $\#\mathcal{F}_{3,1}.$ In fact, we can choose $Q_2$ to be so large that $\#\mathcal{F}'_3$ is at most $0.01\#\mathcal{F}_{3,1}.$ We fix the value for $Q_2.$ Let $\mathcal{F}''_3=\mathcal{F}_{3,1}\setminus\mathcal{F}'_3.$ It is not empty as long as $Q_3\geq Q^M_2.$ 
	
	We can now perform the above argument one more time. Let $Q_4>Q^M_3$ be an integer. We can find $\mathcal{F}_4$ , the collection of intervals of length $2(1/Q_4)^{2+c'}$ centred at each rational numbers with prime denominators in $[Q_4,2Q_4]$ and which intersect $supp(\mu).$ Then we can find the subcollection $\mathcal{F}_{4,1}\subset\mathcal{F}_4$ containing the intervals which are contained in intervals in $\mathcal{F}''_3.$ We have
	\[
	\#\mathcal{F}_{4,1}\gg \#\mathcal{F}''_3 \left(\frac{1}{Q^{2+c'}_3}\right)^s\frac{Q_4^{2s+c'(s-1)}}{\log Q_4}\gg \frac{1}{Q^{c'}_2\log Q_2}\frac{1}{Q^{c'}_3\log Q_3}\frac{Q_4^{2s+c'(s-1)}}{\log Q_4}.
	\]
	Again, we can find a collection $\mathcal{F}'_4$ of intervals in $\mathcal{F}_4$ which intersect intervals in $\mathcal{C}$ with length in between $1/Q^{2+c'}_3$ and $1/Q^{2+c'}_4.$ We see that
	\[
	\#\mathcal{F}'_4\leq \frac{Q^{2s+c'(s-1)}_4}{\log Q_4}\frac{M'}{Q^{(2+c')(s-\rho)}_3}.
	\]
	We can choose $Q_3$ to be large enough so that $\#\mathcal{F}'_4$ is at most $0.01\#\mathcal{F}_{4,1}.$ We now fix the value for $Q_3.$
	
	By repeating the above arguments, we can find integers $Q_4<Q_5<Q_6<Q_7\dots$ and collections of intervals $\mathcal{F}''_5,\mathcal{F}''_6,\mathcal{F}''_7\dots.$ For integers $i<j,$ we have
	\[
	\cup_{I\in\mathcal{F}''_{j}}I\subset\cup_{I\in\mathcal{F}''_{i}}I.
	\] 
	Moreover, for each integer $i,$ the intervals in $\mathcal{F}''_i$ are disjoint with intervals in $\mathcal{C}$ with length at least $1/(Q^{2+c'}_i).$ As the sequence of sets
	\[
	C_i=\cup_{I\in\mathcal{F}''_{j}}I,i\geq 1
	\]
	is a decreasing family of non-empty compact sets,  we see that
	\[
	C_{\infty}=\cap_{i\geq 1}C_i\neq\emptyset.
	\]
	Obviously, $C_\infty\subset W(\psi,0)\cap supp(\mu).$ On the other hand, we see that $C_{\infty}$ cannot intersect non-trivial intervals in $\mathcal{C}.$ If not, let $I\in\mathcal{C}$ such that $I\cap C_{\infty}\neq\emptyset.$ Then for each integer $i\geq 1,$ as $C_{\infty}\subset C_i,$ we see that $|I|<1/(Q^{2+c'}_{i}).$ Therefore $|I|=0$. Thus we see that intervals in $\mathcal{C}$ cannot cover $W(\psi,0)\cap supp(\mu)$ as desired.
	
	Remember that the above arguments are based on (Condition 1). Since $\rho<s-c'/(2+c').$ We see that (Condition 1) is satisfied if
	\[
	\frac{\beta}{1-s}>2+c'.
	\]
	Let $s_1=\dim_{l^1}\mu.$ Recall that we can choose $\beta$ to be smaller but arbitrarily close to $(2s_1-1-c'(1-s_1))/s_1.$ (See (Step 1) in the proof of Theorem \ref{Jarnik2}.) We see that the above can be achieved if
	\[
	\frac{2s_1-1-c'(1-s_1)}{s_1(1-s)}>2+c'.
	\]
	This is the case if
	\[
	ss_1>\frac{1+c'}{2+c'}.
	\]
	We want the above to hold for some $c'>0.$ Thus we need $s_1s>1/2.$ Recall Section \ref{Simple}. We see that for large $p,$ 
	\[
	s_1\geq \frac{1}{2}+\Theta\left(\frac{1}{\log p}\right)
	\]
	and
	\[
	s=\frac{\log (p-t)}{\log p}=1-O\left(\frac{1}{p\log p}\right).
	\]
	Thus we see that 
	\[
	s_1s=\frac{1}{2}+\Theta((\log p)^{-1})-O((p\log p)^{-1})>1/2
	\] 
		for large enough $p.$ This proves the theorem.
\end{proof}

\section{A weak Khinchine divergence result}\label{divergence}

We will prove the following theorem.
\begin{thm}
	Let $\mu$ be a missing digits measure with $\dim_{l^1}\mu>1/2,$ for example, the middle-$15$th Cantor measure. Then $\mu(W(\psi))=1$ for $\psi:q\to 1/(q\log\log q).$
\end{thm}
\begin{rem}
	It is in fact possible to show that the result holds for $\psi:q\to 1/(q (\log q)^\rho)$ for a number $\rho>0$ which is  effectively computable. We decide not to carry out the fully technical arguments because the number $\rho$ is strictly smaller than $1$ and we can not obtain a  fully Khinchine divergence result as in \cite{KL20}. 
\end{rem}
This result improves the results in \cite[Theorem 1.5]{EFS11} and  \cite[Theorem 1.2]{SW19} for the middle $15$th Cantor measure.   Part of our proof uses ideas from \cite{PV2005}. We will give a simple proof for the result with $\psi:q\to \epsilon/q$ and then  provide a slightly more complicated proof which holds for the approximation function $\psi(q)=1/(q\log\log q).$ 
\begin{proof}[Proof for $\psi:q\to \epsilon/q$]
	Let $\epsilon\in (0,1/10).$ First, by Theorem \ref{Lattice Counting},the discussions in Section \ref{LCST} and Example \ref{15th}, we saw that for $\delta_q=\epsilon/q,q\in\mathbb{N},$ there is a constant $c_\mu>0$ and
	\[
	\sum_{q=Q}^{2 Q} \mu(A(\delta_q,q,0))\geq  c_\mu Q\delta_Q
	\]
	for all large enough $Q.$ Of course, this is not quite saying that
	\[
	\mu(\cup_{q=Q}^{2Q} A(\delta_q,q,0))\gg Q\delta_Q.
	\]
	The reason is that the function
	\[
	f_Q=\sum_{q=Q}^{2Q} \chi_{A(\delta_q,q,0)}
	\]
	can attain very large values. Let $Q>10$ be an integer. Let $q,q'$ be two different integers in $[Q,2Q].$ Let $a,a'$ be integers. Suppose that $B(a/q,\epsilon/{q^2})\cap B(a'/q',\epsilon/{q'}^2)$ is not empty. Then we have
	\[
	\left|\frac{a}{q}-\frac{a'}{q'}\right|\leq 2\epsilon\frac{1}{Q^2}.
	\]
	On the other hand, if $a/q\neq a'/q',$ then
	\[
	\left|\frac{a}{q}-\frac{a'}{q'}\right|\geq \frac{1}{4Q^2}.
	\]
	Thus we must have $a/q=a'/q'$ if $B(a/q,\epsilon/{q^2})\cap B(a'/q',\epsilon/{q'}^2)\neq\emptyset.$ Suppose that for a number $\alpha\in [0,1]$ and integers $k',k>1,$ we have
	\[
	f_Q(\alpha)=\sum_{q=Q}^{2Q} \chi_{A(\delta_q,q,0)}(\alpha)=k'\in [k,2k]>1.
	\]
	Then there exist  integers $q_1,\dots,q_{k'}\in [Q,2Q]$ and $a_1,\dots,a_{k'}$ such that
	\[
	\alpha\in B(a_j/q_j,\epsilon/{q_j}^2)
	\]
	for all $j=1,\dots,k'.$ This implies that $a_j/q_j,i\in \{1,\dots,k'\}$ are actually equal, say, to $a_0/q_0$ with $\gcd(a_0,q_0)=1.$ As there are at least $k$ multiples of $q_0$ inside $[Q,2Q],$ we see that $Q\geq (k-1)q_0.$
	
	Now, we want to estimate
	\[
	\int_{f_Q\in [k,2k]} f_Q(\alpha)d\mu(\alpha).
	\]
	From the above arguments, we see that the set $\{f_Q\in [k,2k]\}$ is a union of intervals of length at most $\epsilon/Q^2$ centred at rational numbers with denominators at most $Q/(k-1).$ We claim that
	\[
	\int_{f_Q\in [k,2k]} f_Q(\alpha)d\mu(\alpha)\ll k\left(\frac{Q}{k}\right)^{2s}\left(\frac{\epsilon}{Q^2}\right)^{s}=\epsilon^s/k^{2s-1}.\tag{*}
	\]
	Here, $s=\Haus \mu.$ Indeed, we can list all rational numbers with denominators at most $Q/(k-1).$ There are at most $Q^2/(k-1)^2$ many of them and two different such rational numbers have separation at least $(k-1)^2/Q^2.$ Consider the disjoint union of intervals of length $(k-1)^2/Q^2$ centred at the above rational numbers. We see that among them, at most $\ll (Q/(k-1))^{2s}$ of them can intersect the support of $\mu.$ Now, we shrink those intervals to have length $2\epsilon/Q^2$ and keep their centres fixed.  Thus the total $\mu$ measure of those intervals is 
	\[
	\ll\left(\frac{Q}{k}\right)^{2s}\left(\frac{\epsilon}{Q^2}\right)^{s}.
	\]
	This concludes the claim (*). Here, it is possible to obtain a better estimate than (*) if we employ the GCP(1) for $\mu.$ We will do this in the second proof. Since $2s\geq 2\dim_{l^1}\mu>1,$ by (*), for each $c>0,$ one can fix a large enough $k=k_c$ which depends on $c,\mu,\epsilon$ such that
	\[
	\int_{f_Q\geq k} f_Q(\alpha)d\mu(\alpha)\leq c \epsilon.\tag{**}
	\]
	Now, we see that
	\[
	\mu(\cup_{q=Q}^{2Q} A(\delta_q,q,0))\geq \frac{1}{k}\int_{f_Q<k} f_Q(\alpha)d\mu(\alpha)= \frac{1}{k} \left(\int f_Qd\mu-\int_{f_Q\geq k}f_Qd\mu\right).
	\]
	Observe that
	\[
	\int f_Qd\mu=\sum_{q=Q}^{2Q} \mu(A(\delta_q,q,0)).
	\]
	Thus if $c$ is chosen to be small enough, we have for a number $c',$
	\[
	\mu(\cup_{q=Q}^{2Q} A(\delta_q,q,0))\geq c'\epsilon.
	\]
	for all large enough $Q.$
	This implies that $\limsup_{Q\to\infty} \cup_{q=Q}^{2Q} A(\delta_q,q,0)$ has $\mu$ measure at least $c'\epsilon>0.$ Thus the set $\limsup_{q\to\infty} A(\delta_q,q,0)$ has positive $\mu$ measure. We now upgrade this to full $\mu$ measure.
	
	We need to consider a localized version of (*). In order to do this, we need to study the implied constant in (*) more closely. Let $I\subset [0,1]$ be an interval. Let $m\geq 2$ be an integer. We decompose the unit interval into $m$ equal pieces of length $1/m.$ Let $N^I_m$ be the number of those small intervals that intersect $I\cap supp(\mu).$ We also write $N_m=N^{[0,1]}_m.$ For $r\in (0,1),$ let $\rho(r)$ be the maximal $\mu$-measure of an interval of length $r.$ Then we rewrite (*) as
	\[
	\int_{f_Q\in [k,2k]} f_Q(\alpha)d\mu(\alpha)\leq 2k N_{[Q^2/(k-1)^2]+1} \rho(2\epsilon/Q^2).
	\] 
	Let $l\geq 1$ be an integer. Let $I\subset [0,1]$ be an interval of length $15^{-l}$ such that it is the convex hull of  the support of a branch of $\mu.$ Let $\mu_I=\mu_{|I}/\mu(I).$ Then we see that
	\[
	\int_{f_Q\in [k,2k]} f_Q(\alpha)d\mu_I(\alpha)\leq 2k N^I_{[Q^2/(k-1)^2]+1} \rho(2\epsilon/Q^2)/\mu(I).
	\]
	Next, observe that as long as $Q$ is large enough, 
	\[
	N^I_{[Q^2/(k-1)^2]+1}\leq 240 (Q^2/(k-1)^2 |I|)^s.
	\]
	To see this, we first find the smallest integer $l'$ such that $[Q^2/(k-1)^2]+1\leq 15^{l'}.$ Then we study $N^{I}_{15^{l'}}.$ Suppose that $Q$ is so large that $l'>l.$ Then we have
	\[
	N^{I}_{15^{l'}}=14^{l'-l}=\mu(I) \times 14^{l'}.
	\]
	Here, we have used the fact that $\mu(I)=14^{-l}.$
	Next, for each interval of length $1/([Q^2/(k-1)^2]+1),$ it intersects at most $16$ many disjoint intervals of length $1/15^{l'}.$ From here, we see that
	\[
	N^I_{[Q^2/(k-1)^2]+1}\leq 15\times 16 \mu(I)\times 14^{l'}\leq 15\times 16|I|^s (Q^2/(k-1)^2)^s.
	\]
	Here, we used the fact that $\mu(I)=14^{-l}=|I|^s.$ Next, we consider $\rho(r).$ Clearly, there is a constant $C>0,$ such that $\rho(r)\leq C r^s.$ From here, we see that
	\begin{align*}
		&\int_{f_Q\in [k,2k]} f_Q(\alpha)d\mu_I(\alpha)\leq 2k N^I_{[Q^2/(k-1)^2]+1} \rho(2\epsilon/Q^2)/\mu(I)&\\
		&\leq 480C  k (Q^2/(k-1)^2)^{s} (2\epsilon/Q^2)^s.&\tag{*Loc}
	\end{align*}
	This is a localized version of (*). Let $I$ be the convex hull of a branch of $\mu.$ Consider $\mu_I.$ We  have
	\[
	\sum_{q=Q}^{2 Q} \mu_I(A(\delta_q,q,0))\geq c_{\mu_I}Q\delta_Q.
	\]
	for all large enough $Q.$ The constant $c_{\mu_I}$ does not depend on the interval $I.$ However, the exact form of the quantifier 'large enough' depends on the choice of $I$. This follows from the lower bound in Theorem \ref{Lattice Counting}, Section \ref{LCST} and the fact that $\hat{\mu_I}$ is a scaled version of  $\hat{\mu}.$ Next we use (*Loc) to see that the choice $k_c$ in (**) can be chosen in a uniform manner for all $I$ (convex hulls of branches of $\mu$). Thus, we see that (in what follows the number $c'$ is the same as appeared above)
	\[
	\mu_I(\cup_{q=Q}^{2Q} A(\delta_q,q,0))\geq c'\epsilon.
	\]
	for all large enough $Q$ (in a manner that depends on $I$). This shows that 
	\[
	\mu_I(\limsup_{Q\to\infty} A(\delta_q,q,0))\geq c'\epsilon>0.
	\]
	for all $I$ which is the convex hull of a branch of $\mu.$ Then by \cite[Proposition 1]{BDV ref}, we see that
	\[
	\mu(\limsup_{Q\to\infty} A(\delta_q,q,0))=1.
	\]
	This finishes the proof.
\end{proof}

\begin{proof}[Proof for $\psi:q\to 1/(q\log\log q)$]
	Let $\delta_q=1/(q\log\log q).$ Let $Q>1$ be an integer and we denote 
	\[
	B_Q=\cup_{q=Q}^{2Q} A(\delta_q,q,0).
	\]
	Let $Q_1,Q_2$ be integers. We want to establish the quasi-independence between $f_{Q_1}, f_{Q_2}.$ That is, we want to show that
	\[
	\int f_{Q_1}(x)f_{Q_2}(x)d\mu(x)\ll \int f_{Q_1}(x)d\mu(x)\int f_{Q_2}(x)d\mu(x)
	\] 
	under some conditions on $Q_1,Q_2.$ Suppose that $Q_1<Q_2.$ Then $B_{Q_1}$ is a  union of $\asymp \delta_{Q_1}/{Q_1}$-intervals. As $\mu$ has GCP(1), we see that $\ll \delta_Q/(\delta_Q/Q)^s$ many of them intersect $supp(\mu).$ Let us take $I$ to be one of them. It is possible to find a branch $\mu'$ of $\mu$ such that $supp(\mu')\subset 2I$ and the convex hull of $supp(\mu)$ has length $\gg\delta_{Q_1}/Q_1.$ Then $\mu'$ has GCP(1) and we can estimate $\int f_{Q_2}(x)d\mu'(x).$ We have indeed performed this type of argument in the first two steps of the proof of Theorem \ref{Jarnik2}. In fact, we see that there is number $\beta\in (0,1)$ such that as long as $Q_2\geq (Q_1/\delta_{Q_1})^{1/\beta},$ we have
	\[
	\int f_{Q_2}(x)d\mu'(x)\ll Q_2\delta_{Q_2},
	\]
	where the implied constant depends only on $\mu.$ Our condition on the support of $\mu'$ implies that
	\[
	\int_I f_{Q_2}(x)d\mu(x)\asymp\mu(I)\int_I f_{Q_2}(x)d\mu'(x)\ll (\delta_{Q_1}/Q_1)^s Q_2\delta_{Q_2},
	\]
	where the implied constant depends only on $\mu.$ Thus, we can sum up the above inequality for all intervals $I$ forming $B_{Q_1}$ to see that
	\begin{align*}
		\int f_{Q_1}(x)f_{Q_2}(x)d\mu(x)\ll (Q_1\delta_{Q_1}/(\delta_{Q_1}/{Q_1})^s)(\delta_{Q_1}/Q_1)^s Q_2\delta_{Q_2}=Q_1\delta_{Q_1} Q_2\delta_{Q_2}\\\asymp \int f_{Q_1}(x)d\mu(x)\int f_{Q_2}(x)d\mu(x).
	\end{align*}
	From here we see that as long as $Q_2\geq (Q_1/\delta_{Q_1})^{1/\beta},$
	\[
	\mu(B_{Q_1}\cap B_{Q_2})\ll Q_1\delta_{Q_1} Q_2\delta_{Q_2}.
	\]
	This is not quite achieving the quasi-independence for $B_{Q_1}, B_{Q_2}$ as we have only the following upper bound. 
	\[
	\mu(B_Q)\leq\int f_Q(x)d\mu(x)\ll Q_1\delta_{Q_1}.
	\]
	In order to estimate $\mu(B_{Q}),$ we use the idea for obtaining the inequality (*) in the first proof above (with $\epsilon=1/\log\log Q$). The set $f_Q\in [k,2k]$ is a union of intervals of length at most $1/(Q^2\log\log Q)$ centred at rational numbers with denominators at most $Q/(k-1)$. Since $\mu$ has GCP(1), we see that \[\int_{\{f_Q\in [k,2k]\}} f_Qd\mu\ll \frac{1}{k\log\log Q}\tag{***}\]
	as long as $k\leq Q^{1-c}$ for a constant $c>0$ which depends on $\mu.$ To see this, just observe we have the following estimate
	\[
	\sum_{q\leq Q/({k-1})} \mu(A(q/(Q^2\log\log Q),q,0))\ll \frac{1}{k^2\log\log Q}.\tag{****}
	\]
	This can be done by applying Theorem \ref{Lattice Counting} together with the argument in Section \ref{LCST}. We provide more details. Let $\delta'_{q}=q/(Q^2\log\log Q).$ Consider the function
	\[
	g=\sum_{q=1}^{Q/(k-1)} \chi_{A(\delta'_q,q,0)}.
	\]
	We can modify this function by replacing the character functions with smooth functions just as in the proof of Theorem \ref{Lattice Counting}. We see that it is possible to show that
	\[
	\int gd\mu\ll \int_{[0,1]} g(x)dx+\sum_{|\xi|\leq 4Q^2\log\log Q,\xi\neq 0} |\hat{\mu}(\xi)\hat{g}(\xi)|.
	\]
	The first term can be easily computed and we see that 
	\[
	\int_{[0,1]} g(x)dx\ll \frac{Q^2}{k^2}\frac{1}{Q^2\log\log Q}= \frac{1}{k^2\log\log Q}.
	\]
	For the second term, observe that for $\xi\neq 0,$ we have \[|\hat{g}(\xi)|\leq \sum_{q|\xi,q\leq Q/(k-1)} \delta'_q\leq d(|\xi|)/(kQ\log\log Q).\] We have used the fact that
	\[
	\sum_{q||\xi|,q\leq Q/(k-1)}q\ll \frac{Q}{k} d(|\xi|).
	\]
	We see that ($\dim_{l^1}\mu>1/2$)
	\[
	\sum_{|\xi|\leq 4Q^2\log\log Q}|\hat{\mu}(\xi)|\ll Q^{\rho}
	\]	
	where $\rho<1$ is a constant. Therefore we have
	\[
	\sum_{|\xi|\leq 4Q^2\log\log Q,\xi\neq 0} |\hat{\mu}(\xi)\hat{g}(\xi)|\ll 1/Q^{\rho'}
	\]
	for another constant $\rho'\in (0,1).$ This proves the estimate (****) (and then (***)) for $k\leq Q^{1-c}$ with a suitable constant $c>0$ which can be obtained from $\rho'.$  From (***), we see that for $k\leq Q^{1-c},$
	\[\int_{\{f_Q\geq k\}} f_Qd\mu=\sum_{j\geq 0}\int_{\{f_Q\in [2^j k,2^{j+1}k]\}} f_Qd\mu\ll \frac{1}{k\log\log Q}+\frac{\log Q}{(\log\log Q)^s}\frac{1}{Q^{(2s-1)(1-c)}}.\tag{*****}\]
	For the rightmost inequality, we split the sum of $j$ according to whether $2^j\leq Q^{1-c}$ or not. We apply (***)  for the former case. For the sum with $2^j k\geq Q^{1-c},$ we use (*) from the first proof to see that
	\[
	\int_{\{f_Q\in [2^j k,2^{j+1}k]\}} f_Qd\mu\ll 2^j k \left(\frac{Q}{2^j k}\right)^{2s}\left(\frac{1}{Q^2\log\log Q}\right)^s\ll \frac{1}{(\log\log Q)^s} \left(\frac{1}{Q^{1-c}}\right)^{2s-1}. 
	\]
	Observe that the sum on $j$ is a finite sum because $f_Q(x)\ll Q$ with a uniform implied constant for all $x.$ Thus there are $\ll \log Q$ many summands. From here we see that
	\[
	\sum_{j:2^jk\geq Q^{1-c}}\int_{\{f_Q\in [2^j k,2^{j+1}k]\}} f_Qd\mu\ll \frac{\log Q}{(\log\log Q)^s} \left(\frac{1}{Q^{1-c}}\right)^{2s-1}. 
	\]
	Now (*****) follows. Of course the fact that $2s>1$ is crucial in the above arguments. Now we see that for $k\leq Q^{(2s-1)(1-c/2)}$ we have
	\[
	\int_{\{f_Q\geq k\}} f_Qd\mu\ll \frac{1}{k\log\log Q}.\tag{***'}
	\]
	This is an improvement of (***).  From here, we see that
	\[
	\int_{\{f_Q\geq k\}} f_Qd\mu\ \frac{1}{k\log\log Q}=\frac{Q\delta_Q}{k}.
	\]
	By fixing $k$ to be a large enough integer, we can achieve that
	\[
	\mu(B_Q)\geq \frac{1}{c'k} Q\delta_{Q}
	\]
	for a suitable constant $c'>0.$ From here we see that
	\[
	\mu(B_{Q_1}\cap B_{Q_2})\ll \mu(B_{Q_1})\mu(B_{Q_2})
	\]
	as long as $Q_2\geq Q^{3/\beta}_1.$ We can pick a sequence of integers $Q_1<Q_2<Q_3\dots$ in a way that
	\[
	Q^{3/\beta}_{j}\geq Q_{j+1}\geq Q^{3/\beta}_{j}.
	\] 
	More precisely, we can choose $Q_j=[e^{j\lambda}]$ for a suitable $\lambda>1.$ Then we see that
	\[
	\sum_{j} \mu(B_{Q_j})\gg \sum_{j} \frac{1}{\log j}=\infty.
	\]
	Thus by divergence Borel-Cantelli lemma we see that $\mu(\limsup_{j\to\infty} B_{Q_j})>0.$ To upgrade this result to a full measure version, we can perform the above argument for each branch of $\mu.$ It is already possible to see that for each branch $\mu'$ of $\mu,$ we have $\mu'(\limsup_{j\to\infty} B_{Q_j})>0.$ In order achieve a uniform estimate, one needs to take care of the implied constants appeared in the above $\ll,\gg,\asymp$ symbols. This can be done in the same way as in the first proof. We omit the details. As a result, it is possible to see that there is a number $c''>0$ such that
	\[
	\mu'(\limsup_{j\to\infty} B_{Q_j})>c''
	\]
	for all branch $\mu'$ of $\mu.$ From here we conclude the proof by using \cite[Proposition 1]{BDV ref}.
\end{proof}
\section{Acknowledgement}
HY was financially supported by the University of Cambridge and the Corpus Christi College, Cambridge. HY has received funding from the European Research Council (ERC) under the European Union’s Horizon 2020 research and innovation programme (grant agreement No. 803711). HY thank O. Khalil and M. L\"{u}thi for kindly providing us the preprint \cite{KL20}, and P. Varj\'{u} for many helpful comments.

\bibliographystyle{amsplain}

\end{document}